\documentclass[]{article}
\usepackage{graphicx} 
\usepackage{lipsum}
\usepackage{amsfonts}
\usepackage{amsthm}
\usepackage{graphicx}
\usepackage{algorithm,algorithmicx}
\usepackage{algpseudocode}
\usepackage{bm}
\usepackage{xcolor}
\usepackage{amsmath}
\usepackage{subfig}
\usepackage{bbm}
\usepackage{rotating}
 
\newcommand{\R}{\ensuremath{\mathbb{R}}}
\newcommand{\conv}{\ensuremath{\text{conv}}}
\newcommand{\clconv}{\ensuremath{\text{cl conv}}}
\newcommand{\ignore}[1]{}

\newtheorem{theorem}{Theorem}
\newtheorem{proposition}{Proposition}
\theoremstyle{definition}

\theoremstyle{remark}
\newtheorem{remark}{Remark}
\usepackage{multirow}

\title{Outlier detection in regression: conic quadratic formulations  
\thanks{
{This work was supported by a public grant as part of the Investissement d'avenir project, reference ANR-11-LABX-0056-LMH, LabEx LMH. A G\'omez is supported by grant FA9550-22-1-0369 of the US Air Force Office of Scientific Research.}}}
\author{Andr\'es G\'omez\thanks{Daniel J. Epstein Department of Industrial and Systems Engineering, University of Southern
California, CA 90089 (gomezand@usc.edu).}
\and Jos\'e Neto\thanks{SAMOVAR, Télécom SudParis, Institut Polytechnique de Paris, 91120 Palaiseau, France  
  ({jose.neto@telecom-sudparis.eu}).}}
\date{June 2023}

\begin{document}

\maketitle
\begin{abstract}
In many applications, when building linear regression models, it is important to account for the 
presence of outliers, i.e., corrupted input data points. Such problems can be formulated as mixed-integer optimization problems involving cubic terms, each given by the product of a binary variable and a quadratic term of the continuous variables. Existing approaches in the literature, typically relying on the linearization of the cubic terms using big-M constraints, suffer from weak relaxation and poor performance in practice. In this work we derive stronger second-order conic relaxations that do not involve big-M constraints.
Our computational experiments indicate that the proposed formulations are several orders-of-magnitude faster than existing big-M formulations in the literature for this problem.   
\end{abstract}

%\begin{keywords}
%least trimmed squares, least quantile regression, robust linear regression, outlier detection
%\end{keywords}
\section{Introduction}
\label{sec:pbs_description}

Several statistical and machine learning problems can be formulated as optimization problems of the form
\begin{equation}\label{eq:mipHard}
\begin{aligned}
    \min_{\bm{x},\bm{z}}\;&\sum_{i=1}^m (y_i-\bm{a_i^\top x})^2(1-z_i)\\
    \text{s.t.}\;&(\bm{x},\bm{z})\in F\subseteq \R^n\times \{0,1\}^m,
\end{aligned}
\end{equation}
where $(\bm{a_i},y_i)\in \mathbb{R}^{n+1}$ for all $i\in \{1,\dots,m\}$ are given data and $F$ is the feasible region. Problem \eqref{eq:mipHard} includes the \emph{least trimmed squares} as a special case, which is a focus of this paper and discussed at length in \S\ref{sec:LTS}, but also includes regression trees \cite{dunn2018optimal} (where $1-z_i=1$ indicates that a given datapoint is routed to a given leaf), regression problems with mismatched data \cite{mazumder2023linear} (where variables $\bm{z}$ indicate the datapoint/response pairs) and k-means \cite{kronqvist2022p} (where variables $\bm{z}$ represent assignment of datapoints to clusters). We point out that few or no mixed-integer optimization (MIO) approaches exist in the literature for \eqref{eq:mipHard}, as the problems are notoriously hard to solve to optimality, and heuristics are preferred in practice.

The hardness of problem \eqref{eq:mipHard} is due to weak relaxations such as standard big-M relaxations, 
producing trivial lower bounds of $0$ and gaps of $100\%$. The purpose of this work is thus to derive stronger relaxations of \eqref{eq:mipHard}, paving the way for efficient exact methods via MIO.

\subsection{Robust estimators and least trimmed squares}\label{sec:LTS}
Most statistical methods fail if the input data  is corrupted by so-called {\it{outliers}}. The latter correspond to 
erroneous input data points resulting, e.g., from measurement, transmission, recording errors or exceptional phenomena. 
Consider linear regression models, described by observations $\{(\bm{a_i},y_i)\}_{i=1}^m$ where $\bm{a_i}\in \mathbb{R}^n$ are the features and $y_i$ is the response associated with datapoint $i$. The classical ordinary least squares (OLS) estimator, defined as the minimizer of the optimization problem \begin{equation}\label{eq:OLSSimple}\tag{\tt{OLS}}\min_{\bm{x}\in \R^n}\sum_{i=1}^m(y_i-\bm{a_i^\top x})^2=\min_{\bm{x}\in \R^n}\|\bm{y}-\bm{Ax}\|_2^2,\end{equation}
where $\bm{A}\in \R^{m\times n}$ is the matrix with rows given by $\{\bm{a_i}\}_{i=1}^m$, is known to be sensitive to spurious perturbations of the data. Two robust modifications of \eqref{eq:OLSSimple} are commonly used in practice. The first one calls for the addition of an additional regularization term, resulting in the {\it least squares with Tikhonov regularization} problem. Specifically, given a suitable matrix $\bm{T}$ (typically taken as the identity), the estimator is the optimal solution of
\begin{equation}\label{eq:OLS}\min_{\bm{x}\in \R^n}\sum_{i=1}^m(y_i-\bm{a_i^\top x})^2+\lambda \|\bm{Tx}\|_2^2,\tag{{\tt{LS+L2}}}
\end{equation}
which is robust against small perturbations of the data \cite{el1997robust}. The second approach calls for replacing the least squares loss with the absolute value of the residuals, resulting in \emph{least absolute deviations} (LAD) problem 
\begin{equation}\label{eq:lad}\tag{{\tt{LAD}}}\min_{\bm{x}\in \R^n}\sum_{i=1}^m\left| y_i-\bm{a_i^\top x}\right|.\end{equation}
Estimator \eqref{eq:lad}, which generalizes the median to multivariate regression, is preferred to \eqref{eq:OLSSimple} in settings with outliers. 

Despite their popularity,
\eqref{eq:OLS} and \eqref{eq:lad} are known to be vulnerable to outliers. Robust estimators are often measured according to the breakdown point \cite{HAMPEL1971} -- the smallest proportion of contaminated data that can cause the estimator to take arbitrarily large aberrant values. Clearly, estimators \eqref{eq:OLSSimple} and \eqref{eq:OLS} have an unfavorable breakdown point of $0\%$: a single spurious observation with $\bm{a_i}=\bm{e_j}$, where $\bm{e_j}$ is the $j$-th standard basis vector, and $y_i\to \pm \infty$ will produce solutions where $x_j$ takes arbitrarily bad values. M-estimators \cite{HUBER73,HUBER11}, which include as special cases \eqref{eq:lad} and regression with respect to the Huber loss, also have breakdown point of 0\% \cite{rousseeuw1984robust}.

Robust estimators with better breakdown point include the 
%repeated median estimator and the 
{\it{least median of squares}} ({\tt{LMS}})
\cite{ROUSSEEUW84,ROUSSEEUW87} which minimizes the median of the  
squared residuals. The {\it{least quantile of squares}} ({\tt{LQS}}) \cite{Bertsimas:2014:LQR} approach generalizes 
the latter by minimizing the $q$-th order statistic, i.e., the $q$-th smallest 
residual in absolute value for some given integer $q\leq m$.  
The {\it{least trimmed squares problem}} ({\tt{LTS}}) \cite{ROUSSEEUW84,ROUSSEEUW87}, 
consists in  
minimizing, for some $h\in \mathbb{Z}$, the sum 
of the smallest $h$ residual squares. Specifically, letting $r_i(\bm{x})=|y_i-\bm{a_i^\top x}|$ be the $i$-th residual, and letting $\left|r_{(1)}\left(\bm{x}\right)\right|\leq 
\left|r_{(2)}\left(\bm{x}\right)\right|\leq
\ldots\leq 
\left|r_{(m)}\left(\bm{x}\right)\right|$ be the residuals sorted in nondecreasing magnitude order, the LTS estimator is the optimal solution of
\begin{equation}\label{eq:LTS}
\min_{\bm{x}\in\mathbb{R}^n}
{
\sum_{i=1}^{h}r_{(i)}\left(\bm{x}\right)^2+\lambda\|\bm{Tx}\|_2^2.
}\tag{\tt{LTS+L2}}
\end{equation}
Note that for $h=m$, \eqref{eq:LTS} corresponds to \eqref{eq:OLS}. Intuitively, for $h\leq m-1$, 
the datapoints corresponding to the $m-h$ largest residuals are observations flagged as outliers and discarded prior to using \eqref{eq:OLS} to fit a model on the remaining data. The original {\tt{LTS}} estimator had $\lambda=0$, but we consider here the version with additional $\ell_2$ regularization used in \cite{insolia_2022_outlier_and_feature}, where the additional regularization helps counteracting strong collinearities between features and improves performance in low signal-to-noise regimes.

The {\tt{LMS}} and 
{\tt{LTS}} estimators achieve an optimal breakdown point of 50\% \cite{ROUSSEEUW84,ROUSSEEUW87}. While {\tt{LMS}} was more popular originally, as it is less difficult to compute, Rousseeuw and  Van Driessen \cite{ROUSSEEUW06} argue that ``the {\tt{LMS}} estimator should be replaced by the {\tt{LTS}} estimator" due to several desirable properties, 
including smoothness and statistical efficiency. Unfortunately, computing the {\tt{LTS}} estimator is NP-hard \cite{Bernholt05,BERNHOLT2005b} and even hard to approximate \cite{MOUNT14}.

For the most part, problem \eqref{eq:LTS} is solved using heuristics. In particular, methods which alternate between fitting regression coefficients given a fixed set of $h$ non-outlier observations and determining new outliers given fixed regression coefficients $\bm{\bar x}$ are popular in the literature \cite{HAWKINS1994185,ROUSSEEUW06}.  Solution methods based on solving least trimmed squares with similar iterative approaches have also been proposed in the context of mixed linear regression with corruptions and more general problems, e.g. \cite{NEURIPS2019_c91e3483,pmlr-v97-shen19e} and references therein. Under some specific assumptions on the model, convergence results to an optimal solution have been established for such algorithmic schemes \cite{NIPS2015_bhatia_hard_thresholding}. However, in general, they do not provide guarantees and the quality of the resulting estimators can be poor.

Agull\'o \cite{AGULLO2001425} proposed a branch and bound algorithm to solve \eqref{eq:LTS} to optimality, which is shown to be fast in instances with $m\leq 30$, but struggles in larger instances. A first MIO formulation for \eqref{eq:LTS} was proposed in \cite{giloni2002least}, although the authors observe that the resulting optimization problem is difficult to solve and do not provide computations. To the best of our knowledge, the first implementation of a MIO algorithm for \eqref{eq:LTS} was done in \cite{ZIOUTAS05}, based on a formulation using big-M constraints, where the authors report solution times of two seconds for instances with $m=25$ and also comment on larger computational times for larger instances. In a subsequent work by the same research group \cite{ZIOUTAS09}, the authors report solution times in seconds for problems with $n=2$ and $m\leq 50$, and in minutes for problems with $100\leq m\leq 500$, although all computations are performed on synthetic data. In a recent paper, \cite{insolia_2022_outlier_and_feature} propose another big-M formulation for a generalization of \eqref{eq:LTS} (where sparisty is also imposed on regression variables $\bm{x}$), and report computational times in minutes for synthetic instances with $n$ and $m$ in the low hundreds. We discuss these MIO approaches further in \S\ref{sec:mio}. Finally, we point that exact big-M based MIO algorithms and continuous optimization heuristics were proposed in \cite{Bertsimas:2014:LQR} for the related {\tt{LMS}} problem: the authors report that MIO methods are dramatically outperformed by the continuous optimization approaches, with the objective value of MIO solutions being up to 400x worse than the objective of heuristic solutions (unless the heuristics solutions are used as a warm-start).

\subsection{Contributions, outline and notation}

In this work, we introduce strong, big-M free, conic quadratic reformulations for 
\eqref{eq:LTS} --and, more generally, problems of the form \eqref{eq:mipHard}. 
Extensive 
computational experiments on diverse families of instances (both synthetic and real) clearly point out strong improvements over current state-of-the-art 
approaches.  In particular, the proposed formulations results in orders-of-magnitude improvements over existing big-M formulations in our computations. We refrain from providing an estimate of the scalability of the approach: we show instances with $(n,m)=(20,500)$ that are solved in 10 seconds, and instances with $(n,m)=(4,50)$ that cannot be solved within a time limit of 10 minutes. Indeed, for MIO approaches, the effectiveness of the approach depends on more factors than simply the size of the instance, including the number $m-h$ of observations to be discarded, the regularization parameter $\lambda$, and the overall structure of the dataset (with synthetic instances being considerably easier than real ones). 

The paper is organized as follows. 
We close 
this section with some notation. 
In \S\ref{sec:mio} we review the literature on {\tt{MIO}} 
approaches for linear regression problems. 
Convexification results related to sets originating from \eqref{eq:mipHard} are presented in 
\S\ref{sec:convex}. The convexifications are used to derive 
conic quadratic reformulations of \eqref{eq:LTS} in
\S\ref{sec:lts}.  
 The experimental framework 
 and computational results are presented in  
 \S\ref{sec:computations}
and we conclude the paper in 
 \S\ref{sec:conclusions}.

{\bf{Notation. }} 
For any positive integer $n$, let $[n]$ stand for the set $\{1,2,\ldots,n\}$. 
The vectors and matrices are represented with {\bf{bold}} characters. 
The all-zero and all-one vectors and matrices (with appropriate dimensions) are represented by {$\bm{0}$} and {$\bm{1}$} respectively. 
The $i$-th unit vector is represented by $\bm{e_i}$. 
The notation $\bm{I}$ stands for the identity matrix. 
%Given two vectors ${\bm{u}}$ and ${\bm{v}}$ in $\mathbb{R}^n$, their Hadamard (or entrywise) product is denoted by ${\bm{u}}\circ{\bm{v}}$. 
Given a vector $\bm{d}\in \R^n$, we let $\text{Diag}(\bm{d})\in \R^{n\times n}$ denote the diagonal matrix with elements $\text{Diag}(\bm{d})_{ii}=d_i$. Given a square matrix $\bm{Q}$, we let $\bm{Q}^\dagger$ denote the pseudoinverse of $\bm{Q}$.

\section{Review of MIO methods for outlier detection}\label{sec:mio}
There has been a recent trend of using mathematical optimization techniques to tackle hard problems arising in the context of linear regression. In particular, there is a stream of research focused on the best subset selection problem \cite{atamturk2019rank,atamturk2020safe,ben2022new,bertsimas2016best,bertsimas2020sparse,hazimeh2020fast,hazimeh2022sparse,hazimeh2022l0learn,mazumder2023subset,xie2020scalable,cozad2014learning,cozad2015combined,wilson2017alamo,dong2015regularization}, in which at most $k$ of the regression variables in \eqref{eq:OLS} can take non-zero values. Variants of best subset selection, in which information criteria are used to determine the number of non-zero variables, have also been considered in the literature \cite{kimura2018minimization,park2020subset,miyashiro2015mixed,gomez2021mixed}. Related models have also been used to tackle inference problems with graphical models and sparsity \cite{manzour2021integer,liu2022graph,atamturk2021sparse}. We point out that most of the approaches for sparse regression are based on improving continuous relaxations by exploiting a ridge regularization term $\lambda\|\bm{x}\|_2^2$ through the perspective reformulation \cite{frangioni2006perspective,gunluk2010perspective}. As we show in this paper, the Tikhonov regularization $\lambda\|\bm{Tx}\|_2^2$ is also fundamental for improving relaxations for \eqref{eq:LTS}.

Despite the plethora of {\tt MIO} approaches for sparse regression, there is a dearth of similar methods for regression problems with outliers. Indeed, problems such as \eqref{eq:LTS} appear to be fundamentally more difficult than sparse regression problems. Observe that problem \eqref{eq:LTS} admits the natural mixed-integer cubic formulation \cite{giloni2002least}
\begin{equation}\label{eq:LTS_MIOCubic}
\min_{\bm{x}\in\mathbb{R}^{n},\bm{z}\in \{0,1\}^m}\;
\sum_{i=1}^{m}\left(y_i-\bm{a_i^\top x}\right)^2(1-z_i)+\lambda\|\bm{Tx}\|_2^2 \text{ s.t. } \bm{1^\top z}\leq m-h,
\end{equation}
where $z_i=1$ if datapoint $i$ is flagged as an outlier and discarded, and $z_i=0$ otherwise. Note that \eqref{eq:LTS_MIOCubic} is a special case of \eqref{eq:mipHard}, where $F$ is given by a cardinality constraint. Formulation \eqref{eq:LTS_MIOCubic} cannot be effectively used with most MIO software. Indeed, its natural continuous relaxation, obtained by relaxing the binary constraints to bound constraints $\bm{0}\leq\bm{z}\leq \bm{1}$, is non-convex.
To circumvent this issue, Zioutas et al. \cite{ZIOUTAS05,ZIOUTAS09} reformulated \eqref{eq:LTS_MIOCubic} 
as the convex quadratic mixed integer optimization problem  
\begin{subequations}
\label{eq:lts_qmip}
\begin{align}
%\begin{array}{ll}
%\nonumber
\min_{\bm{u},\bm{x},\bm{z}} &\sum_{i=1}^{m}{u_i^2}+\lambda\|\bm{Tx}\|_2^2\\
%\nonumber
\mathrm{s.t.} &
-y_i+\bm{a_i^{\top}x}\leq u_i+z_i M & \forall i\in [m]\\
%\nonumber
&y_i-\bm{a_i^{\top}x}\leq u_i+z_i M & \forall i\in [m]\\
%\nonumber
&\bm{1}^{\top}\bm{z} \leq m-h\\ 
%\nonumber
&\bm{u}\in\mathbb{R}^m_+, \bm{x}\in\mathbb{R}^{n}, 
\bm{z}\in\{0,1\}^m
%\end{array}
\end{align}
\end{subequations}
where $M$ is a sufficiently large fixed constant and
$u_i$ represents the absolute value of the $i$-th residual.
Indeed, in any optimal solution 
$\left(\bm{u^*},\bm{x^*},\bm{z^*}\right)$ of \eqref{eq:lts_qmip}, having 
$z_i^*=1$ (resp. $z_i^*=0$) implies $u_i^*=0$ 
(resp. $u_i^*=\left|y_i-\bm{a_i}^{\top}\bm{x}\right|$), i.e. the objective value is sum of the squared residuals of the non-outlier datapoints.  

While formulation \eqref{eq:lts_qmip} can be directly used with most {\tt MIO} solvers, the natural continuous relaxation is trivial. Indeed, regardless of the data $(\bm{A},\bm{y},\bm{T})$, an optimal solution of the continuous relaxation is given by $\bm{u^*}=\bm{0}$, $\bm{x^*}=\bm{0}$ and $\bm{z^*}=\left((m-h)/m\right)\bm{1}$. The objective value of this relaxation is thus equal to the trivial lower bound of $0$ (resulting in a 100\% optimality gap), which leads to large branch-and-bound trees as solvers cannot effectively prune the search space. Moreover, the solutions of the continuous relaxations are essentially uninformative, thus {\tt MIO} solvers --which rely on these to produce feasible solutions and inform branching decisions-- struggle to tackle problem \eqref{eq:lts_qmip}. 

In fact, as we show in \S\ref{sec:convex}, any relaxation based on a convex reformulation of the individual cubic terms $\left(y_i-\bm{a_i^\top x}\right)^2(1-z_i)$ necessarily results in trivial bounds and solutions. We point out that this phenomenon sets apart regression problems with outliers from sparse regression problems: the continuous relaxations of the natural big-M formulations of sparse regression problems (e.g., see \cite{bertsimas2016best}) is equivalent to the least squares problem \eqref{eq:OLS}, producing non-trivial bounds and solutions. We conjecture that the difficulty to produce a ``reasonable" convex relaxation of \eqref{eq:LTS_MIOCubic} is the reason why few {\tt MIO} approaches exist for regression problems with outliers.
A notable exception is \cite{doi:10.1137/19M1306233} , which proposes strong conic quadratic formulations for outlier detection with time series data. However, the methodology proposed in that paper is tailored to time series data and cannot be generalized to problem \eqref{eq:LTS}.

\section{Convexification results} 
\label{sec:convex}
In this section we investigate the convex hull of sets related to terms arising in the formulation of problems such as \eqref{eq:LTS}.  
To motivate our approach, 
let us first study the convex hull of the set 
$$
Y_{c} = 
\left\{
\left(\bm{x},z,t\right)\in \mathbb{R}^{n}\times 
\left\{0,1\right\}\times \mathbb{R} \colon 
t\geq 
\left(c-\bm{a}^{\top}\bm{x}\right)^2\left(1-z\right) 
%+ %u\\ u\geq 
%\left(\bm{x}^{\top}\ w\right) \bm{\widehat{D}}
\right\}
$$
where $c\in\mathbb{R}$ is  a scalar. 
$Y_{c}$ may be interpreted as the   
mixed-integer epigraph of the error function associated with a single datapoint. 
As Proposition~\ref{prop:trivialRelax} below shows, any closed convex relaxation of $Y_{c}$ is trivial. 
In other words, any formulation of \eqref{eq:mipHard} based only on convex reformulations of each individual cubic term will result in 100\% gaps and non-informative relaxations. 
\begin{proposition}The closure of the convex hull of 
$Y_{c}$ is given by\label{prop:trivialRelax}
    $$\clconv\left(Y_{c}\right)=\R^n\times [0,1]\times \R_+.$$
\end{proposition}
\begin{proof}
   Consider any point $(\bm{\bar x},\bar z,\bar t)\in \R^n\times [0,1]\times \R_+$ with $0<\bar z<1$. Observe that 
   $$(\bm{\bar x},\bar z,\bar t)=
   \bar z\left(
   \frac{\bm{\bar x}}{\bar z}-\frac{1-\bar z}{\bar z}\frac{c \cdot \bm{a}}{\|\bm{a}\|_2^2},1,0\right)+(1- \bar z) \left(\frac{c\cdot\bm{a}}{\|\bm{a}\|_{2}^{2}},0,\frac{\bar t}{1-\bar z}\right),$$
   where both $\left(
   \frac{\bm{\bar x}}{\bar z}-\frac{1-\bar z}{\bar z}\frac{c \cdot \bm{a}}{\|\bm{a}\|_2^2},1,0\right)\in Y_{c}$ and $\left(\frac{c\cdot\bm{a}}{\|\bm{a}\|_{2}^{2}},0,\frac{\bar t}{1-\bar z}\right)\in Y_{c}$, and thus $(\bm{\bar x},\bar z,\bar t)\in \conv\left(Y_{c}\right)$. Moreover, since $(\bm{\bar x},0,\bar t)=\lim_{z\to 0^+}(\bm{\bar x},z,\bar t)$, we find that $(\bm{\bar x},\bar z,\bar t)\in \clconv\left(Y_{c}\right)$ even if $z=0$.
\end{proof}
%}}

Thus, to derive stronger relaxations, it is necessary to study a more general set, capturing more structural information about the optimization problem. In particular, the formulations we propose to tackle problem \eqref{eq:LTS} are based on a study of the set
\begin{equation}
\label{eq:YcQ_def}
Y_{c,\bm{Q}} =
\left\{
\left(\bm{x},z,t\right)\in \mathbb{R}^{n}\times 
\left\{0,1\right\}\times \mathbb{R} \colon 
t\geq \bm{x}^{\top}\bm{Q}\bm{x}+
\left(c-\bm{a}^{\top}\bm{x}\right)^2\left(1-z\right) 
%+ %u\\ u\geq 
%\left(\bm{x}^{\top}\ w\right) \bm{\widehat{D}}
\right\}
\end{equation}
where 
$\bm{Q}\in \mathbb{R}^{n\times n}$ represents a symmetric and positive definite matrix. 
We provide hereafter descriptions of the convex hull of 
$Y_{c,\bm{Q}}$ in the original space of variables 
%$(\bm{x},\bm{z})$ of formulation \eqref{eq:LTS_MIOCubic} 
for the homogeneous case (i.e., when $c=0$) and in an extended space in the non-homogeneous case ($c\neq 0$). 
 
\subsection{
Convexification for the homogeneous case}\label{sec:convexification_homogeneous}
The convexification of set 
$$
Y_{0,\bm{Q}} =
\left\{
\left(\bm{x},z,t\right)\in \mathbb{R}^{n}\times 
\left\{0,1\right\}\times \mathbb{R} \colon 
t\geq \bm{x}^{\top}\bm{Q}\bm{x}+
\left(\bm{a}^{\top}\bm{x}\right)^2\left(1-z\right) 
\right\}
$$
admits a relatively simple description in the original space of variables. 
\begin{proposition} \label{prop:convYi_}The closure of the convex hull of set $Y_{0,\bm{Q}}$ is

$$
\clconv\left(Y_{0,\bm{Q}}\right)=
\left\{
\left(\bm{x},z,t\right)\in \mathbb{R}^{n}
\times \left[0,1\right]\times \mathbb{R} \colon 
t\geq \bm{x^{\top} Q x} + 
\frac{\left(1-z\right)\left(\bm{a^{\top}x}\right)^2}{1+z\left\|\bm{Q}^{-1/2}\bm{a}\right\|_2^2}
\right\}.
$$
\end{proposition}
\begin{proof}
Let $T$ denote the set in the right-hand side of the equation in the statement of the proposition. We show next that: $\bullet$ $T$ is convex, $\bullet$ $T$ induces a relaxation of $Y_{0,\bm{Q}}$, and $\bullet$ optimization of a linear function over set $T$ is equivalent to optimization over $Y_{0,\bm{Q}}$.

\paragraph{$\bullet$ \textbf{Convexity}} We show convexity of $T$ by establishing it is equivalent to the SDP-representable set given by constraints
$$ 0\leq z\leq 1\text{,  }\begin{pmatrix}\bm{W}&\bm{x}\\\bm{x^\top}&t\end{pmatrix}\succeq 0,\; \bm{W}=\bm{Q^{-1}}-(1-z)\frac{\bm{Q^{-1}aa^\top Q^{-1}}}{1+\|\bm{Q^{-1/2}a}\|_2^2}.$$
Note that $\bm{W}\succ 0$ since for any $\bm{y}\neq \bm{0}$,
$$
%\begin{array}{ll}
\bm{y}^{\top}\bm{W y}\geq
\bm{y^{\top}Q^{\mathrm{-1}} y} 
-\frac{\left(\bm{a^{\top} Q^{\mathrm{-1}} y}\right)^{2}}{1+
\left\|\bm{Q}^{-1/2}\bm{a}\right\|_2^2}
%\end{array}
\geq 
\left(\bm{y}^{\top}\bm{Q}^{-1} \bm{y}\right)\cdot
\left(
1-\frac{\|\bm{Q}^{-1/2}\bm{a}\|_2^2}{1+\|\bm{Q}^{-1/2}\bm{a}\|_2^2}
\right)>0,
$$
where the second inequality uses Cauchy-Schwarz inequality and the 
last one follows from $\bm{y}\neq \bm{0}$ and the definition of $\bm{Q^{-1}}\succ 0$. 
Since $\bm{W}$ is invertible, we find by using the Schur complement \cite{ALBERT69} that $\begin{pmatrix}\bm{W}&\bm{x}\\\bm{x^\top}&t\end{pmatrix}\succeq 0\Leftrightarrow t\geq \bm{x^\top W^{-1}x}$, and using the Sherman Morrison formula \cite{SM49,SM50} we can establish that $\bm{W^{-1}}=\bm{Q}+\frac{1-z}{1+z\|\bm{Q^{-1/2}a}\|_2^2}\bm{aa^\top}.$

\paragraph{$\bullet$ \textbf{Relaxation}} Observe that if $z=0$ then $T$ reduces to the inequality $t\geq \bm{x^\top Qx}+(\bm{a^\top x})^2$, and
if $z=1$ then $T$ reduces to $t\geq \bm{x^\top Qx}$. This is precisely the disjunction encoded by $Y_{0,\bm{Q}}$, hence $T$ is indeed a relaxation.

\paragraph{$\bullet$ \textbf{Equivalence}} 
Now, to prove $T\subseteq \clconv\left(Y_{0,\bm{Q}}\right)$,
 let us consider the optimization of an arbitrary linear function over 
the sets $Y_{0,\bm{Q}}$ and $T$:
\begin{equation}
\label{eq:opt_overYi_}
\min_{\left(\bm{x},z,t\right)\in Y_{0,\bm{Q}}}
{
\bm{\alpha}^{\top}\bm{x}+\beta z+\gamma t
}
\end{equation}
\begin{equation}
\label{eq:opt_overTi_}
\min_{\left(\bm{x},z,t\right)\in T}
{
\bm{\alpha}^{\top}\bm{x}+\beta z+\gamma t
}
\end{equation}
with $\bm{\alpha}\in\mathbb{R}^{n}$, $\beta\in\mathbb{R}$ and 
$\gamma\in\mathbb{R}$.
Obviously if \eqref{eq:opt_overTi_} has an optimal solution 
$\left(\bm{x^*},z^*,t^*\right)$ with 
$z^*\in\{0,1\}$, then it is also an optimal solution for 
\eqref{eq:opt_overYi_}. We then show that whenever 
\eqref{eq:opt_overTi_} admits an optimal solution, there 
exists one with $z$ binary. And if no optimal solution 
exists, then 
both problems \eqref{eq:opt_overYi_}-\eqref{eq:opt_overTi_}  are unbounded. 

\iffalse
\begin{itemize}
\item If $\gamma<0$, then setting $\bm{x}=\bm{0}$, $z=0$ and  
considering $t\rightarrow +\infty$, we see that problems \eqref{eq:opt_overYi_}-\eqref{eq:opt_overTi_} are unbounded. 
\item If $\gamma=0$ and $\bm{\alpha}=\bm{0}$, both problems 
 \eqref{eq:opt_overYi_}-\eqref{eq:opt_overTi_} admit an optimal integral solution of the form 
$\left(\bm{0},z^*,0\right)$ with $z^*\in\{0,1\}$ optimal solution 
of $\min_{z\in [0,1]}{\beta z}$. 
\item If $\gamma=0$ and $\alpha_j\neq 0$ for some $j\in [n]$, consider 
the points of the form 
$\left(\sigma\bm{e_j},1,\lambda_{i,j}\sigma^2\right)$ with 
$\sigma\in\mathbb{R}$. They all belong to the sets 
$Y_{0,\bm{Q}}$ and $T$. Considering then 
$\sigma\rightarrow\pm\infty$ (depending on the sign of $\alpha_j$), we obtain that 
the problems  \eqref{eq:opt_overYi_}-\eqref{eq:opt_overTi_} are unbounded. 
\end{itemize}
\fi

We can assume that $\gamma>0$ since \eqref{eq:opt_overTi_} trivially has a binary solution if $\gamma=0$ and $\bm{\alpha}=\bm{0}$, or both problems are unbounded (for any other combination of parameters with $\gamma\leq 0$).  Moreover, by scaling, we can suppose that $\gamma=1$. 
Then, assume that \eqref{eq:opt_overTi_} has an optimal solution 
$\left(\bm{x}^*,z^*,t^*\right)$ with 
$0<z^*<1$. The point $\left(\bm{x}^*,z^*\right)$ 
is an optimal solution of 
\begin{equation}
\label{eq:auxform1_}
\min_{\left(\bm{x},z\right)\in \mathbb{R}^{n}\times [0,1]} 
q\left(\bm{x},z\right)
\end{equation}
with 
\begin{equation}
\label{eq:expression_q_}
q(\bm{x},z)=\bm{\alpha}^{\top}\bm{x}+\beta z+
\left\|\bm{Q}^{\mathrm{1/2}}\bm{x}\right\|_{2}^{2} + 
\frac{\left(1-z\right) \left(\bm{a}^{\top}\bm{x}\right)^2}{1+
z \left\|\bm{Q}^{\mathrm{-1/2}}\bm{a}\right\|_{2}^{2}}.
\end{equation}
Fixing $z$ in \eqref{eq:expression_q_} and using the first order 
optimality conditions, we deduce the following 
expression of an optimal solution $\bm{x}(z)$ of 
$\min_{\bm{x}\in \mathbb{R}^{n}} 
q\left(\bm{x},z\right)$:
\begin{equation}
\label{eq:expression_xz_}
\bm{x}(z)=-\frac{1}{2} \bm{Q}^{\mathrm{-1}} \bm{\alpha} + 
\frac{1-z}{2\left(
1+\left\|\bm{Q}^{\mathrm{-1/2} a}\right\|_{2}^{2}
\right)} 
\bm{Q}^{-1} \bm{a} \bm{a}^{\top} \bm{Q}^{-1} 
\bm{\alpha}.
%\left[
%\bm{I}-\frac{1-z}{\lambda_i+\left\|\bm{\hat{a}_i}\right\|_{2}^{2}}
%\bm{\hat{a}_i\hat{a}_i^{\top}}
%\right]\bm{\alpha}.
\end{equation}
Thus, problem \eqref{eq:auxform1_} reduces to 
$
\min_{z\in [0,1]} 
q\left(\bm{x}(z),z\right).
$
Substituting $\bm{x}(z)$ by its expression \eqref{eq:expression_xz_} in 
\eqref{eq:expression_q_}, we obtain 
that $q\left(\bm{x}(z),z\right)$ is a linear function of $z$. 
To be more precise, after computations, we get the following expression. 
\begin{equation}
q\left(\bm{x}(z),z\right)=\beta z -
\frac{1}{4}\left\|\bm{Q}^{\mathrm{-1/2}}\bm{\alpha}\right\|_{2}^{2}+
\frac{ \left(\bm{a}^{\top}\bm{Q}^{\mathrm{-1}}\bm{\alpha}\right)^2}{4\left(
1+\left\|\bm{Q}^{-1/2}\bm{a}\right\|_{2}^{2}\right)}\left(1-z\right). 
\end{equation}
Thus, 
\eqref{eq:auxform1_} admits an optimal solution with 
$z\in\{0,1\}$, concluding the proof. 
\end{proof}

Intuitively, since terms $(1-z)(\bm{a^\top x})^2$ do not admit a good convex reformulation (Proposition~\ref{prop:trivialRelax}), the key is to instead use the \emph{non-convex} reformulation 
$r(\bm{x})=\frac{\left(1-z\right)\left(\bm{a^{\top}x}\right)^2}{1+z\left\|\bm{Q}^{-1/2}\bm{a}\right\|_2^2}$.
To illustrate, consider the case with $n=1$ and $a=1$, that is, 
\begin{align*}
    Y_{0,\lambda} &=
\left\{
\left(x,z,t\right)\in \mathbb{R}\times 
\left\{0,1\right\}\times \mathbb{R} \colon 
t\geq \lambda x^2+
x^2\left(1-z\right) 
%+ %u\\ u\geq 
%\left(\bm{x}^{\top}\ w\right) \bm{\widehat{D}}
\right\},\text{ and}\\
\clconv\left(Y_{0,\lambda}\right)&=
\left\{
\left(x,z,t\right)\in \mathbb{R}
\times \left[0,1\right]\times \mathbb{R} \colon 
t\geq \lambda x^2 + 
\frac{\left(1-z\right)x^2}{1+z/\lambda}
\right\},
\end{align*}
where $\lambda>0$ is a parameter that controls the magnitude of the quadratic term. Figure~\ref{fig:graph} (top) depicts the graphs of the convex envelopes $t= \lambda x^2 + 
\frac{\left(1-z\right)x^2}{1+z/\lambda}$ for various values of $\lambda$. Moreover, Figure~\ref{fig:graph} (bottom) depicts the graphs of the non-convex reformulation $r= \frac{\left(1-z\right)x^2}{1+z/\lambda}$ for the associated values of $\lambda$. Note that $r$ can also be interpreted as the quantity added to the relaxation induced by big-M relaxations such as \eqref{eq:lts_qmip}, which discard terms associated with $x^2(1-z)$ altogether. We observe that larger improvements over big-M relaxations are achieved for larger values of parameter $\lambda$.

\begin{figure}[!ht]
	\centering
\subfloat[\texttt{$t= 0.1 x^2 + 
\frac{\left(1-z\right)x^2}{1+z/0.1}$}]{\includegraphics[width=0.33\textwidth,trim={13cm 5.5cm 13cm 5.5cm},clip]{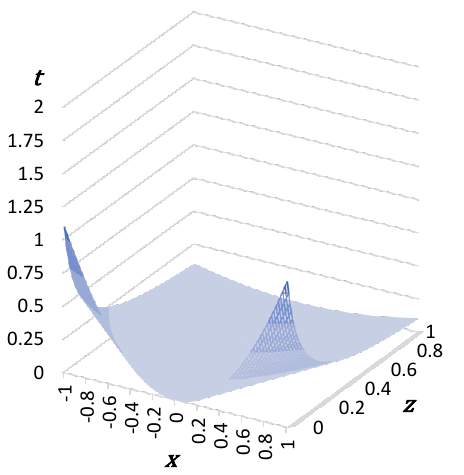}}\hfill\subfloat[\texttt{$t= 0.5 x^2 + 
\frac{\left(1-z\right)x^2}{1+z/0.5}$}]{\includegraphics[width=0.33\textwidth,trim={13cm 5.5cm 13cm 5.5cm},clip]{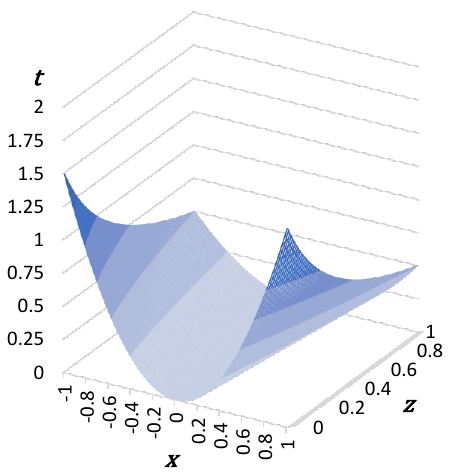}}\hfill	\subfloat[\texttt{$t= 1.0 x^2 + 
\frac{\left(1-z\right)x^2}{1+z/1.0}$}]{\includegraphics[width=0.33\textwidth,trim={13cm 5.5cm 13cm 5.5cm},clip]{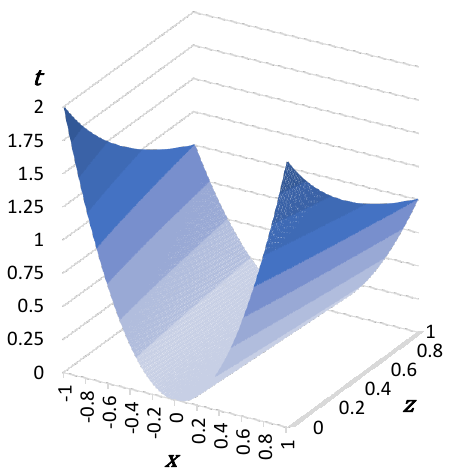}}\hfill\newline %\vskip 5mm 
\subfloat[\texttt{$r=\frac{\left(1-z\right)x^2}{1+z/0.1}$}]
{\includegraphics[width=0.33\textwidth,trim={13cm 5.5cm 13cm 5.5cm},clip]{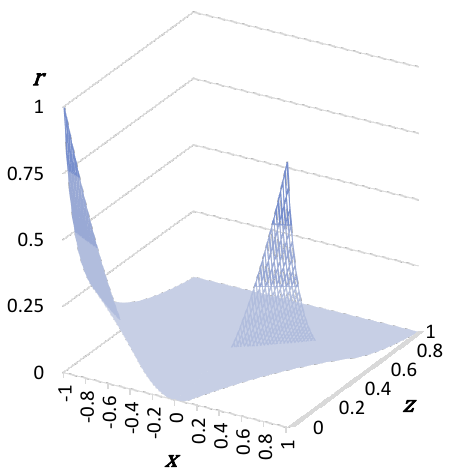}}\hfill\subfloat[\texttt{$r=\frac{\left(1-z\right)x^2}{1+z/0.5}$}]{\includegraphics[width=0.33\textwidth,trim={13cm 5.5cm 13cm 5.5cm},clip]{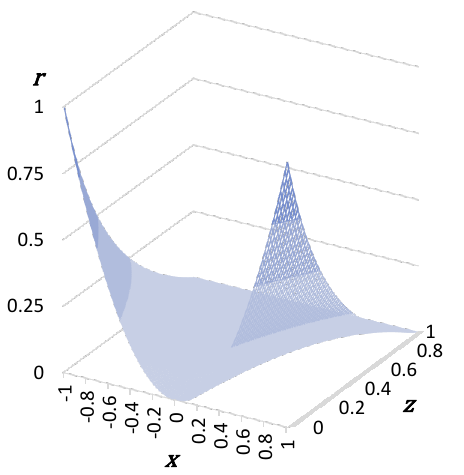}}\hfill	\subfloat[\texttt{$r=\frac{\left(1-z\right)x^2}{1+z/1.0}$}]{\includegraphics[width=0.33\textwidth,trim={13cm 5.5cm 13cm 5.5cm},clip]{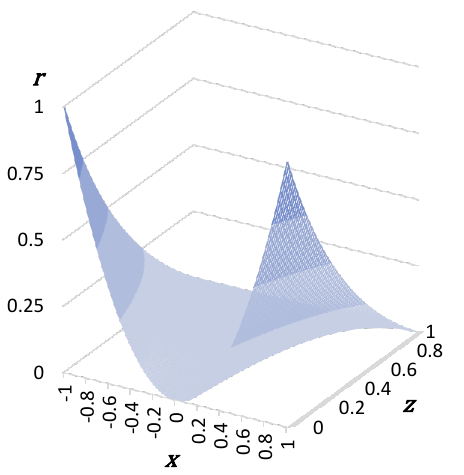}}\hfill\newline %\vskip 5mm 
	\caption{\small Graphs of the convex envelopes $t= \lambda x^2 + 
\frac{\left(1-z\right)x^2}{1+z/\lambda}$ (top) and the non-convex reformulation $r=
\frac{\left(1-z\right)x^2}{1+z/\lambda}$ (bottom) for $\lambda\in \{0.1,0.5,1.0\}$. Note that while the reformulation induced by $r$ is non-convex, the convex envelope is due the strict convexity of term $\lambda x^2$.}
	\label{fig:graph}
\end{figure}

\subsection{
Convexification for the general case}
We now consider the non-homogeneous case  where $c\neq 0$. We could not establish a simple description of $\clconv (Y_{c,\bm{Q}})$ in the original space of variables. Moreover, while relaxations of $Y_{c,\bm{Q}}$ can be derived from Proposition~\ref{prop:convYi_} by writing $Y_{c,\bm{Q}}=
\left\{
\left(x_0,\bm{x},z,t\right)\in \mathbb{R}^{n+1}\times 
\left\{0,1\right\}\times \mathbb{R} \colon 
t\geq \bm{x}^{\top}\bm{Q}\bm{x}+
\left(cx_0-\bm{a}^{\top}\bm{x}\right)^2\left(1-z\right), x_0=1 
%+ %u\\ u\geq 
%\left(\bm{x}^{\top}\ w\right) \bm{\widehat{D}}
\right\}$, we found in preliminary computations that the resulting convexifications (which do not account for constraint $x_0=1$) could be much weaker. Fortunately, as we show in this section, $\clconv(Y_{c,\bm{Q}})$ admits an easy representation with the introduction of an additional variable. 

Observe that set $Y_{c,\bm{Q}}$ can be written as projection onto the $(\bm{x},z,t)$ space of \begin{align*}
\hat Y_{c,\bm{Q}} =
\Big\{
\left(\bm{x},z,t,w\right)\in \mathbb{R}^{n}\times 
\left\{0,1\right\}\times \mathbb{R}^2 \colon &
t\geq \bm{x}^{\top}\bm{Q}\bm{x}+
\left(c+w-\bm{a}^{\top}\bm{x}\right)^2,\\ &w(1-z)=0
\Big\}.
\end{align*}
Indeed, if $z=0$, then $w=0$ and $ Y_{c,\bm{Q}}$ and $ \hat Y_{c,\bm{Q}}$ coincide. On the other hand, if $(\bm{x},1,t)\in Y_{c,\bm{Q}}$, then $(\bm{x},1,t,\bm{a^\top x}-c)\in \hat Y_{c,\bm{Q}}$. We now characterize $\text{cl conv}(\hat Y_{c,\bm{Q}})$. Let $\bm{L}\in \R^{n\times n}$ be any matrix such that $\bm{LL^\top}=\left(\bm{Q}+\bm{aa^\top}\right)^{-1}$, obtained for example from a Cholesky decomposition.

\begin{theorem} \label{prop:convYc} The closure of the convex hull of set $\hat Y_{c,\bm{Q}}$ is
\begin{align*}
&\clconv\left(\hat Y_{c,\bm{Q}}\right)=
\Bigg\{
\left(\bm{x},z,t,w\right)\in \mathbb{R}^{n}
\times \left[0,1\right]\times \mathbb{R}^2 \colon\\
&\;\;t\geq c^2+2c(w-\bm{a^\top x})
+\left\|\bm{L^{-1}}\left(\bm{x}-\frac{\bm{Q^{-1}a}}{1+\|\bm{Q}^{-1/2}\bm{a}\|_{2}^{2}}w\right)\right\|_2^2+\frac{w^2}{\left(1+\|\bm{Q^{-1/2}a}\|_2^2\right)z}
\Bigg\}.
\end{align*}
\end{theorem}
\begin{proof}
In the proof, first we compute $\clconv\left(\hat Y_{c,\bm{Q}}\right)$ in an SDP-representable extended formulation, then we simplify to a lower-dimensional SOCP-representable set, and finally we project out all additional variables.

\paragraph{SDP-representable formulation}
    Observe that \begin{equation}\label{eq:QPform}\bm{x}^{\top}\bm{Q}\bm{x}+
\left(c+w-\bm{a}^{\top}\bm{x}\right)^2=c^2+2c\left(w-\bm{a}^{\top}\bm{x}\right) + %u\\ u\geq 
\left(\bm{x}^{\top}\ w\right) \bm{Q_1}
\begin{pmatrix}
\bm{x} \\
w
\end{pmatrix}\end{equation}
with $\bm{Q_1}=
%\begin{pmatrix}
\left(
\renewcommand{\arraystretch}{1.2}
\begin{array}{l|l}
\bm{Q+a a^{\top}} & -\bm{a}\\ 
\hline 
-\bm{a^{\top}} & 1
\end{array}
\right).
%\end{pmatrix}
$ Define $
\bm{Q_0}=\left(
 \renewcommand{\arraystretch}{1.2}
\begin{array}{l|l}
%\begin{pmatrix}
\bm{Q+a a^{\top}}  & \bm{0}\\ \hline
\bm{0^{\top}} & 0
%\end{pmatrix}
\end{array}
\right)
$. Then a description of $\clconv\left(\hat Y_{c,\bm{Q}}\right)$ in an extended formulation is  \cite{WEI22}
\begin{equation}\label{eq:hullSDP}
\begin{aligned}
\clconv\left(\hat Y_{c,\bm{Q}}\right)=
\Big\{
\left(\bm{x},z,t,w\right)\in \mathbb{R}^{n+3}
 \colon& \exists \bm{W}\in \R^{(n+1)\times (n+1)},\tau\in \R \text{ s.t. }\\
&t\geq c^2+2c\left(w-\bm{a}^{\top}\bm{x}\right) + 
\tau\\
&\left(\begin{array}{c c c}
\tau & \bm{x^\top}&w \\
\bm{x} & \multicolumn{2}{c}{\multirow{2}{*}{$\bm{W}$}}\\
w&\multicolumn{2}{c}{}\\
\end{array}\right)\succeq 0\\
&(z,\bm{W})\in \conv(P)
\Big\},
\end{aligned}
\end{equation}
where $P=\left\{(0,\bm{Q_0^\dagger}), (1,\bm{Q_1^\dagger}) \right\}$ and $\bm{Q_i}^\dagger$ denotes the pseudoinverse of $\bm{Q_i}$. Clearly, $\conv(P)=\{(z,\bm{W})\in [0,1]\times \R^{(n+1)\times (n+1)}: \bm{W}=(1-z)\bm{Q_0^\dagger}+z\bm{Q_1^\dagger}\}.$ 

\paragraph{SOCP-representable formulation}
Note that expressions of $\bm{Q_0^\dagger}$ and $\bm{Q_1^\dagger}$ can be easily computed \cite{LU2002119}:
\begin{align*}
&\bm{Q_0^\dagger}=\left(\begin{array}{l|l}
%\begin{pmatrix}
\left(\bm{Q+a a^{\top}}\right)^{-1} & \bm{0}\\ \hline
\bm{0^{\top}} & 0
%\end{pmatrix}
\end{array}\right)=\left(
\begin{array}{c|c}
\bm{Q}^{-1}- \frac{\bm{Q}^{-1}\bm{a a^{\top}} \bm{Q}^{-1}}{1+\|\bm{Q}^{-1/2}\bm{a}\|_{2}^{2}}
 & \bm{0}\\
 \\
\hline\\
\bm{0}^{\top}\ &\ 0
\end{array}
\right)\\
&\bm{Q_1^\dagger}=\bm{Q_1^{-1}}=\left(
\begin{array}{c|c}
\bm{Q^{-1}}
 & \  \bm{Q}^{-1}\bm{a}\\
 \\
\hline\\
 \bm{a^{\top}}\bm{Q}^{-1}\ &\  1+\|\bm{Q}^{-1/2}\bm{a}\|_{2}^{2}
\end{array}
\right).
\end{align*}

Therefore, we find that constraint $\bm{W}=(1-z)\bm{Q_0^\dagger}+z\bm{Q_1^\dagger}$ simplifies to 
%\begin{subequations}
\begin{align*}
\bm{W}
&=
\left(
\begin{array}{c|c}
\bm{Q}^{-1}- \frac{\bm{Q}^{-1}\bm{a a^{\top}} \bm{Q}^{-1}}{1+\|\bm{Q}^{-1/2}\bm{a}\|_{2}^{2}}
 & \bm{0}\\
 \\
\hline\\
\bm{0}^{\top}\ &\ 0
\end{array}
\right)+
z\left(
\begin{array}{c|c}
\frac{\bm{Q}^{-1}\bm{a a^{\top}} \bm{Q}^{-1}}{1+\|\bm{Q}^{-1/2}\bm{a}\|_{2}^{2}}
 & \  \bm{Q}^{-1}\bm{a}\\
 \\
\hline\\
 \bm{a^{\top}}\bm{Q}^{-1}\ &\  1+\|\bm{Q}^{-1/2}\bm{a}\|_{2}^{2}
\end{array}
\right)\\ 
\\
&=\bm{U}+ z \bm{vv^{\top}}
\end{align*}
where
\begin{align*}
\bm{U}= 
\left(
\begin{array}{c|c}
\bm{Q}^{-1}- \frac{\bm{Q}^{-1}\bm{a a^{\top}} \bm{Q}^{-1}}{1+\|\bm{Q}^{-1/2}\bm{a}\|_{2}^{2}}
 & \bm{0}\\
 \\
\hline\\
\bm{0}^{\top}\ &\ 0
\end{array}
\right)
\mbox{ and } \bm{v}=
\left(
\begin{array}{l}
\frac{\bm{Q}^{-1}\bm{a}}
{\sqrt{1+\left\|\bm{Q}^{-1/2}\bm{a}\right\|_{2}^{2}}}\\ \\
{\tiny\sqrt{1+\left\|\bm{Q}^{-1/2}\bm{a}\right\|_{2}^{2}} }
\end{array}
\right).
\end{align*}

Moreover, the system \begin{equation}\label{eq:system}\exists \bm{W}\in \R^{(n+1)\times (n+1)}\text{ s.t. }\left(\begin{array}{c c c}
\tau & \bm{x^\top}&w \\
\bm{x} & \multicolumn{2}{c}{\multirow{2}{*}{$\bm{W}$}}\\
w&\multicolumn{2}{c}{}\\
\end{array}\right)\succeq 0,\; \bm{W}=\bm{U}+z\bm{vv^\top}
\end{equation}can be reformulated as an SOCP \cite[p.227-229]{NESTEROV94}. 
Letting $\bm{L}\in\mathbb{R}^{n\times n}$ such that 
$\bm{L L^{\top}}= \left(\bm{Q}+\bm{aa^{\top}}\right)^{-1}$   
(obtained for example from a Cholesky decomposition), 
then point $(\bm{x},z,w)$ satisfies constraints \eqref{eq:system} if and only if there exists $\tau_1,\tau_2\in\mathbb{R}_+$, $s\in\mathbb{R}$ and 
$\bm{u}\in\mathbb{R}^{n}$ such that the constraints
\begin{equation}
\begin{aligned}\label{eq:SOCP_system}
&\tau=  \tau_1+\tau_2\\
&\bm{L u} + \frac{\bm{Q^{-1}a}}{\sqrt{1+\|\bm{Q}^{-1/2}\bm{a}\|_{2}^{2}}}s = 
\bm{x}\\
&s\sqrt{1+\left\|\bm{Q}^{-1/2}\bm{a}\right\|_{2}^{2}} =w\\
&\left\|\bm{u}\right\|_{2}^{2} \leq {\tau_1}\\
&s^2 \leq {\tau_2 z}
\end{aligned}
\end{equation}
are satisfied.

\paragraph{Projection}
In system \eqref{eq:SOCP_system}, we can project out $\tau=\tau_1+\tau_2$, $s=w/\sqrt{1+\left\|\bm{Q}^{-1/2}\bm{a}\right\|_{2}^{2}}$ and $\bm{u}=\bm{L^{-1}}\left(\bm{x}-\frac{\bm{Q^{-1}a}}{1+\|\bm{Q}^{-1/2}\bm{a}\|_{2}^{2}}w\right)$, which by replacing in \eqref{eq:hullSDP} results in the formulation
\begin{equation}\label{eq:hullSOCP}
\begin{aligned}
\clconv\left(\hat Y_{c,\bm{Q}}\right)=
\Big\{
\left(\bm{x},z,t,w\right)\in \mathbb{R}^{n+3}
 \colon& \exists \tau_1,\tau_2\in \R_+ \text{ s.t. }\\
&t\geq c^2+2c\left(w-\bm{a}^{\top}\bm{x}\right) + 
\tau_1+\tau_2\\
&\left\|\bm{L^{-1}}\left(\bm{x}-\frac{\bm{Q^{-1}a}}{1+\|\bm{Q}^{-1/2}\bm{a}\|_{2}^{2}}w\right)\right\|_2^2\leq \tau_1\\
&w^2/\left(1+\|\bm{Q}^{-1/2}\bm{a}\|_{2}^{2}\right)\leq \tau_2 z
\Big\}.
\end{aligned}
\end{equation}
In order to satisfy the first inequality constraint, we can assume that $\tau_1$ and $\tau_2$ are set to their lower bounds, concluding the proof.
\end{proof}

\begin{remark}\label{rem:perspective}
Theorem~\ref{prop:convYc} reveals an interesting connection between $\clconv (\hat Y_{c,\bm{Q}})$ and the perspective reformulation. Indeed, letting $\bm{e_{n+1}}$ denote the $(n+1)$th standard basis vector of $\R^{n+1}$, one can rewrite the quadratic expression in \eqref{eq:QPform}~as
$$\left(\bm{x}^{\top}\ w\right) \bm{Q_1}
\begin{pmatrix}
\bm{x} \\
w
\end{pmatrix}=\delta w^2+\left(\bm{x}^{\top}\ w\right) \left(\bm{Q_1}-\delta \bm{e_{n+1}e_{n+1}^\top}\right)
\begin{pmatrix}
\bm{x} \\
w
\end{pmatrix},$$
where $\delta\geq 0$ and $\bm{Q_1}-\delta\bm{e_{n+1}e_{n+1}^\top}\succeq 0$, and then reformulate term $\delta w^2$ as $\delta w^2/z$. From Theorem~\ref{prop:convYc}, we see that this reformulation is indeed ideal if $\delta$ is maximal, and the theorem provides a closed-form expression for the resulting $\bm{Q_1}-\delta \bm{e_{n+1}e_{n+1}^\top}$ (that depends on the factorization $\bm{LL^\top}$). 
\end{remark}

\section{Application to LTS}
\label{sec:lts}

In this section, we use the convexification results in \S\ref{sec:convex} to obtain conic reformulations of \eqref{eq:LTS}, or equivalently, problem \eqref{eq:LTS_MIOCubic}. 

\subsection{The Big-M formulation}\label{sec:bigM} The starting point for the formulations presented is the big-M formulation
\begin{equation}
\label{eq:lts_bigM}\tag{Big-M}
\begin{aligned}
%\begin{array}{ll}
%\nonumber
\min_{\bm{x},\bm{z},\bm{w}} &\sum_{i=1}^{m}\left(y_i+w_i-\bm{a_i^\top x}\right)^2+\lambda\|\bm{Tx}\|_2^2\\
%\nonumber
\mathrm{s.t.}\; &\bm{1^\top z}\leq m-h\\
&-M\bm{z}\leq\bm{w}\leq M\bm{z}\\
&\bm{x}\in\mathbb{R}^{n},\; 
\bm{z}\in\{0,1\}^m,\;\bm{w}\in \R^m,
\end{aligned}
\end{equation}
where $M$ is a suitably large number. Observe that while the formulation is different from the original big-M formulation \eqref{eq:lts_qmip} proposed in \cite{ZIOUTAS05}, they are equivalent in terms of strength. Indeed, variables $\bm{u}$ in \eqref{eq:lts_qmip} corresponds to terms $|y_i+w_i-\bm{a_i^\top x}|$ in \eqref{eq:lts_bigM}, and the absolute values of variables $\bm{w}$ in \eqref{eq:lts_bigM} can be interpreted as the slacks associated with constraints $|y_i-\bm{a_i^\top x}|-u_i\leq Mz_i$ in \eqref{eq:lts_qmip}. We point out that formulation \eqref{eq:lts_bigM} was the basis for the solution approach in \cite{insolia_2022_outlier_and_feature} for problems with both outliers and sparsity. Indeed, the authors proposed to directly add constraints of the form $-\bar M \bm{\zeta}\leq \bm{x}\leq \bar M\bm{\zeta}$, $\bm{1^\top \zeta}\leq k$ and $\bm{\zeta}\in \{0,1\}^n$ to \eqref{eq:lts_bigM} -- note that in \cite{insolia_2022_outlier_and_feature}, the regularization term $\lambda\|\bm{Tx}\|_2^2$ appeared in as constraint instead of as a penalty. 

\subsection{The simple conic reformulation}\label{sec:conicSimple}
Observing that the objective of \eqref{eq:lts_bigM} can be written as $$\sum_{i=1}^m\left((y_i+w_i-\bm{a_i^\top x})^2+\frac{\lambda}{m}\|\bm{Tx}\|_2^2\right),$$
we use Theorem~\ref{prop:convYc} to independently reformulate each term in the sum, resulting in the formulation
\begin{equation}\tag{conic}
\label{eq:lts_conic}
\begin{aligned}
%\begin{array}{ll}
%\nonumber
\min_{\bm{x},\bm{z},\bm{w}} &\|\bm{y}\|_2^2-2\bm{y^\top}\left( \bm{A x}-\bm{w}\right)+\sum_{i=1}^m\left\|\bm{L_i^{-1}}\left(\bm{x}-\frac{(m/\lambda)(\bm{T^\top T})^{-1}\bm{a}}{1+(m/\lambda)\|(\bm{T^\top T})^{-1/2}\bm{a}\|_{2}^{2}}\cdot w_i\right)\right\|_2^2\\
&\quad +\sum_{i=1}^m\frac{1}{1+(m/\lambda)\|(\bm{T^\top T})^{-1/2}\bm{a}\|_{2}^{2}}\cdot\frac{w_i^2}{z_i}\\
%\nonumber
\mathrm{s.t.}\; &\bm{1^\top z}\leq m-h\\
&\bm{x}\in\mathbb{R}^{n},\; 
\bm{z}\in\{0,1\}^m,\;\bm{w}\in \R^m,
\end{aligned}
\end{equation}
where matrices $\bm{L_i}$ satisfy $\bm{L_iL_i^\top}=\left((m/\lambda)\bm{T^\top T}+\bm{a_ia_i^\top}\right)^{-1}$. Formulation \eqref{eq:lts_conic} does not use the big-M constraints $-M\bm{z}\leq \bm{w}\leq M\bm{z}$, as the conic terms $w_i/z_i$ enforce the same logical relationship. Observe that since terms $x_i^2/z_i$ can be reformulated as SOCP-constraints \cite{akturk2009strong}, and every other term is either convex quadratic or linear, formulation \eqref{eq:lts_conic} can be easily used with mixed-integer SOCP solvers.

\subsection{The stronger conic reformulation}\label{sec:conicStronger}

The observation motivating the stronger conic reformulation is that, given any collection of matrices $\{\bm{Q_i}\}_{i=1}^m$ such that $\bm{Q_i}\succ 0$ and $\sum_{i=1}^m\bm{Q_i}=\lambda\bm{T^\top T}$, we can rewrite the objective of \eqref{eq:lts_bigM} as
$$\sum_{i=1}^m\left((y_i+w_i-\bm{a_i^\top x})^2+\bm{x^\top Q_ix}\right)$$ and then apply Theorem~\ref{prop:convYc}. The simple conic reformulation is a special case of such a convexification, with $\bm{Q_i}=(\lambda/m)\bm{T^\top T}$ for all $i\in [m]$, but other choices of collection $\{\bm{Q_i}\}_{i=1}^m$ may result in stronger formulations. We now discuss how to find a collection $\{\bm{Q_i}\}_{i=1}^m$ resulting in better relaxations.

We use the intuition provided in Remark~\ref{rem:perspective} and similar ideas to \cite{dong2015regularization,zheng2014improving} to derive the formulation. Observe that given any collection $\{\bm{Q_i}\}_{i=1}^m$, the relaxation is of the form
\begin{subequations}
\begin{align*}\min_{\bm{x},\bm{z},\bm{w}} &
\left\|\bm{y}\right\|_{2}^{2}-2\bm{y}^{\top}
\left(\bm{Ax}-\bm{w}\right)+\begin{pmatrix}\bm{x^\top} &\bm{w}^\top\end{pmatrix}\bm{\Sigma}\begin{pmatrix}\bm{x}\\ \bm{w}\end{pmatrix} +
\sum_{i=1}^{m}{ d_i\frac{w_{i}^{2}}{z_i}}\\
\nonumber \mbox{s.t. }
& \sum_{i=1}^{m}{z_i}\leq m-h\\
\nonumber &\bm{x}\in\mathbb{R}^{n}, \bm{z}\in [0,1]^m, 
\nonumber \bm{w}\in\mathbb{R}^m,
\end{align*}
\end{subequations}
where $\bm{\Sigma}\succeq 0$ and $\bm{d}\geq \bm{0}$. Moreover, we find that $$\bm{\Sigma}=\begin{pmatrix}\bm{A^\top A}+\lambda \bm{T^\top T}& -\bm{A^\top}\\-\bm{A}&\bm{I}-\text{Diag}(\bm{d})\end{pmatrix}.$$
Thus, the continuous relaxation of the stronger conic reformulation is given by 
\begin{equation}
\label{eq:lts_conicP}\tag{conic+}
\begin{aligned}
\max_{\bm{d},\bm{\Sigma}}&\min_{\substack{(\bm{x},\bm{z},\bm{w})\in \R^{n+2m}\\
\|\bm{z}\|_1\leq m-h\\
\bm{0}\leq\bm{z}\leq \bm{1}}} 
\left\|\bm{y}\right\|_{2}^{2}-2\bm{y}^{\top}
\left(\bm{Ax}-\bm{w}\right)+\begin{pmatrix}\bm{x^\top} &\bm{w}^\top\end{pmatrix}\bm{\Sigma}\begin{pmatrix}\bm{x}\\ \bm{w}\end{pmatrix} +
\sum_{i=1}^{m}{ d_i\frac{w_{i}^{2}}{z_i}}\\
 \mbox{s.t. }
& \bm{\Sigma}=\begin{pmatrix}\bm{A^\top A}+\lambda \bm{T^\top T}& -\bm{A^\top}\\-\bm{A}&\bm{I}-\text{Diag}(\bm{d})\end{pmatrix}\succeq 0\\
&\bm{d}\in \R_+^m,\;\bm{\Sigma}\in\mathbb{R}^{(n+m)\times(n+m)}.
\end{aligned}
\end{equation}
Observe that in formulation \eqref{eq:lts_conicP}, constraints $\bm{z}\in\{0,1\}$ were relaxed to bound constraints, hence it is a relaxation of \eqref{eq:LTS_MIOCubic}. The MIO version corresponds to fixing $\bm{d}$ and $\bm{\Sigma}$ to the optimal values of \eqref{eq:lts_conicP}, and adding back constraints $\bm{z}\in \{0,1\}^m$. 

We now discuss how to compute an optimal solution $\bm{d^*}$ of \eqref{eq:lts_conicP} -- note that $\bm{\Sigma^*}$ is immediately implied from the value of $\bm{d^*}$.
Given any fixed $(\bm{\bar x},\bm{\bar z},\bm{\bar w})\in \R^n\times [0,1]^m\times \R^m$, the choice of $\bm{d}$ that results in the best relaxation for that particular point (i.e., resulting in the largest objective value for that particular point with $\bm{z}$ fractional) is an optimal solution of the semidefinite optimization problem (where we remove terms that do not depend on $\bm{d}$)
\begin{subequations}\label{eq:sdp1}
\begin{align}
    \max_{\bm{d}\in \R_+^m}\;&\sum_{i=1}^m \bar w_i^2\left(\frac{1}{\bar z_i}-1\right)d_i\\
\text{s.t.}\;&\begin{pmatrix}\bm{A^\top A}+\lambda \bm{T^\top T}& -\bm{A^\top}\\-\bm{A}&\bm{I}-\text{Diag}(\bm{d})\end{pmatrix}\succeq 0.\label{eq:sdp1_psd}
\end{align}
\end{subequations}

While convex and polynomial-time solvable, problem \eqref{eq:sdp1} can be difficult to solve, mainly due to the presence of the large-dimensional conic constraint \eqref{eq:sdp1_psd}, on order $(n+m)$ matrices. Fortunately, as Proposition~\ref{prop:sdpReformulation} below shows, problem \eqref{eq:sdp1} can be reformulated using a lower dimensional conic constraint, on order $n$ matrices.

\begin{proposition}\label{prop:sdpReformulation}
If $\bm{A}$ does not contain a row of $0$s and $\bm{u^*}$ is optimal for the optimization problem
\begin{equation}\label{eq:decomposition}
\begin{aligned}
    \min_{\bm{u}\in \R^m}\;&\sum_{i=1}^m\bar w_i^2\left(\frac{1}{\bar z_i}-1\right)\frac{1}{u_i}\\
\text{s.t.}\;&\bm{A^\top A}+\lambda\bm{T^\top T}-\bm{A^\top}\text{Diag}(\bm{u})\bm{A} \succeq 0,\; \bm{u}\geq \bm{1},
\end{aligned}
\end{equation}
then $\bm{d^*}\in \R^m_+$ such that $d_i^*=1-\frac{1}{u_i^*}$ is optimal for \eqref{eq:sdp1}.
\end{proposition}
\begin{proof}
From the generalized Schur complement \cite{ALBERT69}, we find that constraint \eqref{eq:sdp1_psd} is equivalent to 
\begin{subequations}\label{eq:schur}
\begin{align}&\bm{I}-\text{Diag}(\bm{d})\succeq 0,\label{eq:schur_1}\\
&\left(\bm{I}-\text{Diag}(\bm{d})\right)\left(\bm{I}-\text{Diag}(\bm{d})\right)^\dagger\bm{A}=\bm{A},\text{ and }\label{eq:schur_2}\\
&\bm{A^\top A}+\lambda\bm{T^\top T}-\bm{A^\top}\left(\bm{I}-\text{Diag}(\bm{d})\right)^{\dagger}\bm{A} \succeq 0.\label{eq:schur_3}
\end{align}
\end{subequations}
Constraint \eqref{eq:schur_1} is equivalent to $\bm{d}\leq \bm{1}$. Constraint \eqref{eq:schur_2} is automatically satisfied if $\bm{d}<\bm{1}$, since in that case matrix $\left(\bm{I}-\text{Diag}(\bm{d})\right)^\dagger =\left(\bm{I}-\text{Diag}(\bm{d})\right)^{-1}$. In general, however, $\Omega=\left(\bm{I}-\text{Diag}(\bm{d})\right)\left(\bm{I}-\text{Diag}(\bm{d})\right)^\dagger$ is the diagonal matrix such that $\Omega_{ii}=\mathbbm{1}_{\{d_i<1\}}$. Therefore, if $d_i=1$, then the $i$-th row of matrix $\Omega A$ is a row of $0$s, and constraint \eqref{eq:schur_2} cannot be satisfied in that case unless the $i$-th row of $\bm{A}$ is also $\bm{0}$. 

Finally, perform a change of variables $u_i=\frac{1}{1-d_i}$, well defined since $d_i<1$ holds. From constraints $\bm{d}\geq 0$ we find $\bm{u}\geq \bm{1}$. Problem \eqref{eq:sdp1} reduces to 
\begin{align*}
    \max_{\bm{u}\in \R^m}\;&\sum_{i=1}^m\left( \bar w_i^2\left(\frac{1}{\bar z_i}-1\right)-\bar w_i^2\left(\frac{1}{\bar z_i}-1\right)\frac{1}{u_i}\right)\\
\text{s.t.}\;&\bm{A^\top A}+\lambda\bm{T^\top T}-\bm{A^\top}\text{Diag}(\bm{u})\bm{A} \succeq 0,\; \bm{u}\geq \bm{1}.
\end{align*}
The result then follows by removing terms in the objective not involving~$\bm{u}$.
\end{proof}

\begin{remark}
The assumption on $\bm{A}$ is almost without loss of generality, since it is invariably satisfied in practice. In the formulation in Proposition~\ref{prop:sdpReformulation}, the nonlinear objective terms can reformulated with the introduction of additional variables $\bm{s}\in \R_+^m$ and rotated cone constraints $1\leq s_iu_i$. The formulation contains a similar number of variables as \eqref{eq:sdp1}, but if $n\ll m$ the nonlinear conic constraints are much simpler, and as a consequence the resulting formulation is substantially faster (and less memory intensive as well). 
\end{remark}

We propose a simple primal-dual method to solve \eqref{eq:lts_conicP}, summarized in Algorithm~\ref{alg:primal_dual}. The algorithm iterates between solving the inner minimization of \eqref{eq:lts_conicP} to optimality (for fixed $\bm{d}$ and $\bm{\Sigma}$), and moving towards the optimal of the outer maximization (for fixed $\bm{x}$, $\bm{z}$ and $\bm{w}$). Each minimization step requires solving an SOCP-representable problem, while  each maximization step requires solving an SDP as outlined in Proposition~\ref{prop:sdpReformulation}. The final MIO can be solved with off-the-shelf mixed-integer SOCP solvers.

\begin{algorithm}[h]
	\caption{\eqref{eq:lts_conicP} algorithm.}
	\label{alg:primal_dual} \small
	\begin{algorithmic}[1]
		\renewcommand{\algorithmicrequire}{\textbf{Input:}}
		\renewcommand{\algorithmicensure}{\textbf{Output:}}
            \State $k\leftarrow 0$ \Comment{Iteration number}
            \State $d_i^0\leftarrow \frac{1}{1+(m/\lambda)\left\|(\bm{T^\top T})^{-1/2}\bm{a_i}\right\|_2^2}$ for $i=1,\dots,m$\Comment{$\bm{d^0}$ from formulation \eqref{eq:lts_conic}}
            \Repeat
            \State $k\leftarrow k+1$
            \State $(\bm{\bar x},\bm{\bar z},\bm{\bar w})\leftarrow$ \texttt{Solve} inner minimization of \eqref{eq:lts_conicP} with $\bm{d}=\bm{d^{k-1}}$ fixed\label{line:relax}
            \State $\bm{d^*} \leftarrow$ \texttt{Solve} \eqref{eq:sdp1} \Comment{Use Proposition~\ref{prop:sdpReformulation}}
            \State $\bm{d^{k}}\leftarrow \bm{d^{k-1}}+\frac{1}{k}\left(\bm{ d^*}-\bm{d^{k-1}}\right)$
            \Until{Termination criterion is met}
            \State $(\bm{x^*},\bm{z^*},\bm{w^*})\leftarrow$ \texttt{Solve} \eqref{eq:lts_conicP} with constraints $\bm{z}\in \{0,1\}^m$ and $\bm{d}=\bm{d^k}$ fixed.
		\State \Return $(\bm{x^*},\bm{z^*},\bm{w^*})$
	\end{algorithmic}
\end{algorithm}

\subsection{Improving relaxations with reliable data}\label{sec:reliable}
In some situations, a decision-maker may have access to data that has been carefully vetted, and is known to be reliable. Obviously, in such situations, such data should not be discarded. Moreover, as we now discuss, it is possible to leverage such data to further improve the relaxations.

Suppose that the first $m_0$ datapoints are known to not contain outliers. Then, \eqref{eq:lts_bigM} simplifies to 
\begin{align*}
%\begin{array}{ll}
%\nonumber
\min_{\bm{x},\bm{z},\bm{w}} &\sum_{i=1}^{m_0}\left(y_i-\bm{a_i^\top x}\right)^2+\sum_{i=m_0+1}^{m}\left(y_i+w_i-\bm{a_i^\top x}\right)^2+\lambda\|\bm{Tx}\|_2^2\\
%\nonumber
\mathrm{s.t.}\; &\bm{1^\top z}\leq m-h\\
&-M\bm{z}\leq\bm{w}\leq M\bm{z}\\
&\bm{x}\in\mathbb{R}^{n},\; 
\bm{z}\in\{0,1\}^{m-m_0},\;\bm{w}\in \R^{m-m_0}.
\end{align*}
Expanding the error terms of first $m_0$ points, we may rewrite the objective as  
$$\sum_{i=1}^{m_0}\left(y_i^2-2y_i\bm{a_i^\top x}\right)+\sum_{i=m_0+1}^{m}\left(y_i+w_i-\bm{a_i^\top x}\right)^2+\bm{x^\top}\left(\lambda \bm{T^\top T}+\sum_{i=1}^{m_0}\bm{a_ia_i^\top}\right)\bm{x}.$$
In other words, matrix $\lambda \bm{T^\top T}+\sum_{i=1}^{m_0}\bm{a_ia_i^\top}$ can be treated as the ``regularization" matrix and used throughout the conic formulations instead of $\lambda \bm{T^\top T}$, resulting in stronger formulations.

\begin{remark}
    Note that even if no reliable data is available, the ideas here can still be used to improve algorithms. For example, in a branch-and-bound search, at any given node some subset of variables $\bm{z}$ may have been fixed to $\bm{0}$. Thus, we may use the ideas in this subsection to improve the relaxations for the subtree emanating from that node. Doing so, however, would require a large degree of control of the branch-and-bound algorithm, which is not possible for several off-the-shelf branch-and-bound solvers.
\end{remark}
 
\subsection{Intercept}\label{sec:intercept}
The presence of the strictly convex term $\|\bm{Tx}\|_2^2$ is critical for the design of strong convex relaxations 
(Proposition~\ref{prop:trivialRelax}). However, the presence of an intercept variable might hamper the exploitation of the regularization term. Indeed, while the intercept is often subsumed into matrix $\bm{A}$ (as a column of 1s), the regularization term rarely involves the intercept variable. Indeed, writing \eqref{eq:lts_bigM} while making the intercept variable $x_0$ explicit results in 
\begin{equation*}
\begin{aligned}
%\begin{array}{ll}
%\nonumber
\min_{x_0,\bm{x},\bm{z},\bm{w}} &\sum_{i=1}^{m}\left(y_i+w_i-x_0-\bm{a_i^\top x}\right)^2+\lambda\|\bm{Tx}\|_2^2\\
%\nonumber
\mathrm{s.t.}\; &\bm{1^\top z}\leq m-h\\
&-M\bm{z}\leq\bm{w}\leq M\bm{z}\\
&(x_0,\bm{x})\in\mathbb{R}^{n+1},\; 
\bm{z}\in\{0,1\}^m,\;\bm{w}\in \R^m,
\end{aligned}
\end{equation*}
where the intercept $x_0$ does not appear in term $\|\bm{Tx}\|_2^2$, and thus this term is not strictly convex. Observe that the quadratic objective function is rank-deficient, as the rank of the quadratic function is at most $m+n$, while there are $n+m+1$ variables $(x_0,\bm{x},\bm{w})$. As a consequence the conic formulations may be ineffective, e.g., feasible solutions of optimization \eqref{eq:sdp1} may require $d_i=0$ for at least some index $i$.

We propose three workarounds to resolve the difficulties posed by the intercept. The first one, which is the ideal solution, is to use reliable data as discussed in \S\ref{sec:reliable}: a single datapoint known to be reliable will allow for the application of the conic formulations. The second approach, which is common in practice, is to standardize $\bm{y}$ so that it has $0$-mean, and fix $x_0=0$. In other words, fix the intercept to be the mean of the response variable. The third approach is to artificially create a strictly quadratic term involving the intercept as a regularization. In particular, given a baseline value $\bar c_0$ for the intercept (e.g., obtained from a heuristic solution), add a regularization term $\lambda (x_0-\bar c_0)^2$
to the objective, penalizing departure of $x_0$ from this baseline. Naturally, the addition of this regularization may prevent the resulting formulation from finding optimal solutions of \eqref{eq:LTS_MIOCubic}. Nonetheless, in our computations, we found that the solutions obtained from the conic formulations are still high quality even if the baseline value $\bar c_0$ is  poorly chosen.

\section{Computations}
\label{sec:computations}

In this section, we discuss computations with synthetic and real data. First, in \S\ref{sec:comp_methods}, we discuss the different methods compared. Then in \S\ref{sec:summary} we provide a high level summary of our computational results, in \S\ref{sec:comp_synt} we discuss experiments with synthetic data, used to validate the statistical merits of the approach, and in \S\ref{sec:comp_real} we provide experiments with real datasets. 

\subsection{Methods tested}\label{sec:comp_methods}

We compare the three MIO formulations presented in \S\ref{sec:lts} for regression problems with outliers. 
\begin{description}
\item[\texttt{big-M}] The big M method, as discussed in \S\ref{sec:bigM}. 
\item[\texttt{conic}] The simple conic reformulation, as discussed in \S\ref{sec:conicSimple}.
\item[\texttt{conic+}] The stronger conic reformulation, as discussed in Algorithm~\ref{alg:primal_dual} in \S\ref{sec:conicStronger}.
\ignore{
\item[\texttt{lqs-mio}] The \texttt{MILO} formulation for the least quantile of squares (\texttt{LQS}) given in \cite{Bertsimas:2014:LQR}. Recall that the \texttt{LQS} minimizes the $h$-th order statistic, i.e., solves the optimization problem
\begin{equation}\label{eq:LQS}
\min_{\bm{x}\in\mathbb{R}^n}
{
r_{(h)}\left(\bm{x}\right)^2
}
\end{equation}
where $\left|r_{(1)}\left(\bm{x}\right)\right|\geq 
\left|r_{(2)}\left(\bm{x}\right)\right|\geq
\ldots\geq 
\left|r_{(m)}\left(\bm{x}\right)\right|$ are the residuals sorted in nonincreasing magnitude order. The formulation, which is based on big-M constraints, is described in Appendix~\ref{sec:lqs}.}
\end{description}

In addition, we also compare the following commonly used methods.
\begin{description}
\item[\texttt{ls+l2}] Simply solving problem \eqref{eq:OLS}, as described in \S\ref{sec:pbs_description}, without accounting for outliers.
\item[\texttt{lad}] Solving the least absolute deviation problem \eqref{eq:lad}.
\item[\texttt{alt-opt}] Heuristic that alternates between optimizing regression coefficients $\bm{x}$ given a fixed set of $m-h$ discarded observations (by fitting an \texttt{ls+l2} regression), and optimizes which $m-h$ observations to discard (encoded by $\bm{z}$) given fixed regression coefficients (a process called a C-step in \cite{ROUSSEEUW06}). We set the initial regression coefficients to be those obtained from \eqref{eq:OLS}.
\end{description}

We point out that we also attempted to implement the MIO formulation of \cite{Bertsimas:2014:LQR} for least quantile regression. However that formulation, which includes three different sets of big-M constraints, resulted in numerical issues in most of the instances (with the solver producing ``optimal" solutions that are not feasible for the MIO, and are extremely poor estimators). In any case, as mentioned in \S\ref{sec:pbs_description}, the authors in \cite{Bertsimas:2014:LQR} comment that the solutions produced by the MIO formulation are much worse than heuristic solutions (unless warm-started with such solutions, in which case the quality of solutions produced by MIO matches the heuristics). 

In all cases we use the ridge regularization $\bm{T}=\bm{I}$.
We use Gurobi solver 9.5 to solve all (mixed-integer) linear, quadratic or second order cone optimization problems, and solver Mosek 10.0 to solve SDPs. All computations are done on a laptop with a 12th Gen Intel Core i7-1280 CPU and 32 GB RAM. We set a time limit of 10 minutes for all methods (for \texttt{conic+}, this time includes both solving SDPs and a MIO), and use the default configuration of the solvers in all cases. All instances are standardized, that is, the model matrix $\bm{A}$ is translated and scaled so that $\sum_{i=1}^m A_{ij}=0$ and $\sum_{i=1}^m A_{ij}^2=1$ for all $j\in [n]$; similarly, $\sum_{i=1}^m y_{i}=0$ and $\sum_{i=1}^m y_{i}^2=1$.

\subsubsection{Implementation details and numerical considerations} We now discuss how we select and tune parameters for the different methods, as well as discuss potential issues if the parameters are poorly chosen. 

\paragraph{\texttt{big-M}} Formulation \eqref{eq:lts_bigM} depends on the parameter $M$. If a small value is chosen, then the formulation might remove optimal solutions. If the value chosen is too large, then numerical issues can be encountered: for example, the solver might set $z_j=10^{-5}$ for some $j$ (which is interpreted as $0$ due to numerical precision of solvers) but set $w_j$ to a large value, while satisfying constraint $|w_j|\leq Mz_j$. In our experiments we set $M=1,000$, and we did not observe any numerical issue in our experiments. This parameter was not tuned.

\paragraph{\texttt{conic}} Formulation \eqref{eq:lts_conic} does not involve any parameter, however it requires to have $\lambda>0$ and may result in incorrect behavior if $\lambda\to 0$. Moreover, based on past experience by the authors, mixed-integer SOCP formulations may result in poor performance or numerical difficulties in large problems. Our experiments satisfy $\lambda\geq 0.01$ and we did not observe any numerical issues. 

To handle the intercept (recall the discussion in \S\ref{sec:intercept}), we tested both fixing $x_0=0$, or using the intercept $\bar x_0$ produced by \texttt{ls+l2} as a proxy, and adding the regularization term $\lambda (\bar x_0-x_0)^2$ to the objective (where $\lambda$ is the same coefficient as the one appearing in term $\lambda\|\bm{x}\|_2^2).$ We did not observe major differences between the two approaches, in terms of quality or solution times. In our experiments with real data we set $x_0=0$, and in our experiments with synthetic data we use the value of \texttt{ls+l2} as a proxy (not that the synthetic experiment includes the instances where \texttt{ls+l2} performs worse).

\paragraph{\texttt{conic+}} As the most sophisticated formulation, there are several implementation details associated with method \texttt{conic+}. First, note that if an optimal solution of problem \eqref{eq:decomposition} satisfies $u_i^*=1$ for some index $i\in [m]$, then $d_i^*=0$ and formulation \eqref{eq:lts_conicP} does not include term $w_i^2/z_i$ (thus the MIO formulation is not exact, but rather a relaxation). Thus, in our computations, we set a constraint $u_i\geq 1.001$. We did not tune this lower bound, although we noted that simply setting $u_i\geq 1$ does indeed result in incorrect results. 

We now discuss the termination criterion of Algorithm~\ref{alg:primal_dual}. Note that at each iteration, at line~\ref{line:relax} of the algorithm, a lower bound on the optimal objective value of \eqref{eq:LTS_MIOCubic} is computed. Moreover, given the solution to the relaxation $(\bm{\bar x}, \bm{\bar z},\bm{\bar w})$, we can compute an upper bound on the optimal objective value using a rounding heuristic, by setting $z_j=1$ for indexes corresponding to the $m-h$ largest values of $\bm{z}$, and then solving \eqref{eq:LTS_MIOCubic} with $\bm{z}$ fixed. Neither the sequence of lower and upper bounds produced by the algorithm is guaranteed to be monotonic, so we track the best lower (LB) and upper bound (UB) found throughout all previous iterations, and compute an optimality gap at any given iteration as $\text{gap}=(UB-LB)/UB$. Finally, we stop the algorithm after 20 iterations (not necessarily consecutive) in which the  gap improvement from one iteration to the next is less than $10^{-6}$. The parameters $(20, 10^{-6})$ were tuned minimally, based on one synthetic instance and one real instance with the alcohol dataset (and on these datasets we did not observe major differences for different choices of parameters).  

In terms of numerical difficulties, in addition to those already mentioned for the \texttt{conic} formulation, method \texttt{conic+} requires solving several SDPs with low-dimensional cones. Certainly, SDPs are inherently more difficult than quadratic or conic quadratic problems, more sensitive to the input data and more prone to numerical instabilities. For example, we observed that if the raw data $(\bm{A},\bm{y})$ is used without standardization, the SDP solver encounters numerical difficulties in several of the instances. In our experiments, with standardized data, we did encounter numerical issues in a single instance (out of over 400). The intercept is handled similarly to \texttt{conic}. 

\ignore{\paragraph{\texttt{lqs-mio}} The MIO formulation proposed in \cite{Bertsimas:2014:LQR} uses three sets of big-M constraints and, similar to method \texttt{big-M}, is susceptible to error if the big-M values are poorly chosen. We use the approach suggested in \cite{Bertsimas:2014:LQR} to handle the big-M constraint, by modeling them as SOS1 constraints (which essentially lets the solver decide the big-M value to be used). We often encountered numerical issues (discussed in detail in \S\ref{sec:comp_synt}), with the solver returning solutions which it claims as optimal, but are in fact infeasible for the MIO. }

\subsubsection{A note on additional improvements} 

We point out that all methods presented here can be further improved. For example, one might use the solution of any of the heuristic methods as warm start for the MIO, and after run heuristic \texttt{alt-opt} starting from the solution produced by the MIO (assuming the time limit was reached, in which case the solution of the MIO might not be optimal), which will produce a solution that is at least as good as the solutions obtained from either solving the MIO or using heuristic \texttt{alt-opt} independently. We do suggest practitioners to use such improvements in practice.
However, we point out that our objective in the paper is not to propose an algorithm that is ``best" for the LTS problem, but rather to evaluate the strength of the conic formulations presented (which, as pointed out in \S\ref{sec:pbs_description}, might be used as building blocks for other optimization problems). By not including additional improvements, we ensure that the differences in computational performance between the MIO methods is entirely due to the formulations used.

\subsection{Summary of results}\label{sec:summary}

We first provide a summary of the results in our computations.

$\bullet$ In computations with both synthetic and real datasets, formulations \texttt{conic} and \texttt{conic+} are orders-of-magnitude faster than \texttt{big-M}. 

$\bullet$ In computations with both synthetic and real datasets, heuristic \texttt{alt-opt} is very fast and delivers optimal solutions in a good proportion of instances, but produces extremely poor solutions in the remaining instances (often worse than  solutions obtained by not handling outliers at all). The exact formulation \texttt{conic+} consistently produces high-quality solutions (even in instances that are not solved to optimality).

$\bullet$ In our computations, formulation \texttt{conic+} solves all synthetic instances in a few seconds --including instances with $(n,m)\in \{20,500\}$-- but fails to solve within the time limit real instances with $(n,m)\in \{10,60\}$. Therefore, we advocate  for departing from the common practice in the statistical and machine learning literature to use synthetic instances to evaluate the scalability of exact methods for \eqref{eq:LTS} or related problems. 

\subsection{Experiments with synthetic data}\label{sec:comp_synt}

We now discuss experiments with synthetic data. First we discuss the instance generation process in \S\ref{sec:synt_gen} and the relevant metrics in \S\ref{sec:synt_metrics}, then we present computational performance in \S\ref{sec:synt_time} and statistical results in \S\ref{sec:synt_res}. We point out that the  %main point 
focus in this section is  the statistical performance of the different methods.  

\subsubsection{Instance generation}\label{sec:synt_gen}

Our instance generation process follows closely the generation process in \cite{Bertsimas:2014:LQR}, which in turn was inspired by \cite{ROUSSEEUW06}. Given parameters $n$ and $m$, each entry of the matrix $\bm{A}$ is generated iid as $A_{ij}\sim \mathcal{N}(0,100)$. Moreover, we generate a ``ground-truth" vector $\bm{x}^*=\bm{1}$, and %generate  
responses as $\bm{y}=\bm{Ax}+\bm{\epsilon}$, where each entry of $\bm{\epsilon}$ is generated iid as $\epsilon_i\sim\mathcal{N}(0,10)$.  %Finally, 
Given a proportion $\tau$ of outliers, we randomly choose $\lfloor \tau m\rfloor$ of the observations as outliers and modify the associated responses as $y_i\leftarrow y_i+1,000$. Finally, data is standardized. 

In all our experiments, we set the budget of outliers $m-h$ to be equal to the proportion of outliers $\lfloor \tau m\rfloor$ (as done in \cite{Bertsimas:2014:LQR,ROUSSEEUW06}).

\subsubsection{Metrics}\label{sec:synt_metrics}

In this section we compare the statistical benefits of the different methods. To assess the quality of a given method, we compare two metrics. Given an estimate $\bm{\hat x}$ for a given method, the relative risk defined as $$\texttt{risk}=\frac{\|\bm{x^*}-\bm{\hat x}\|_2^2}{\|\bm{x^*}\|_2^2},$$
where $\bm{x^*}=\bm{1}$ is the ground truth, measures the error of the estimated coefficients. A perfect prediction results in $\texttt{risk}=0$, the naive predition $\bm{\hat x}=\bm{0}$ results in $\texttt{risk}=1$, and method with low breakdown point (e.g., \texttt{ls+l2}) may result in arbitrarily large values of \texttt{risk}. Given a set of candidate outliers encoded by an indicator vector $\bm{\hat z}$, the recall defined as 
$$\texttt{recall}=\frac{\left|\{i\in [m]: \hat z_i=1\text{ and }i \text{ is an outlier}\}\right|}{\lfloor \tau m\rfloor}$$
measures the proportion of outliers that are correctly identified. Note that \texttt{recall} is not defined for \texttt{lad} and \texttt{ls+l2}, since those methods do not explicitly identify outliers. 

\subsubsection{Computational performance}\label{sec:synt_time}
We test small instances for all combinations of parameters $n\in \{2,20\}$, $m\in\{100,500\}$, $\lambda\in \{0.01,0.1,0.2,0.3\}$ and $\tau\in \{0.1,0.2,0.4\}$. For each combination of parameters, we generate five instances. Heuristics such as \texttt{lad} and \texttt{alt-opt} are solved in a fraction of a second. The results for MIO formulations are summarized in Figure~\ref{fig:performance_synt}. We observe that method \texttt{big-M} can solve 50\% of the instances in under 10 minutes, a performance comparable with the one reported in \cite{insolia_2022_outlier_and_feature} (although the data generation process is different). Methods \texttt{conic} and \texttt{conic+} are much faster. In particular, \texttt{conic+} can solve all instances in less than 11 seconds, resulting in at least a two-orders-of-magnitude speedup over~\texttt{big-M}.

\begin{figure}[!ht]
	\centering
\includegraphics[width=0.80\textwidth,trim={11cm 6cm 11cm 6cm},clip]{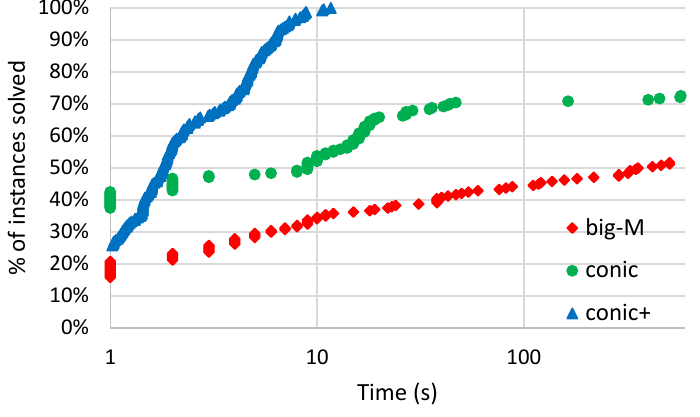}
	\caption{\small Percentage of synthetic instances solved as a function of time (in log scale). Method \texttt{big-M} can solve 52\% of instances in 600 seconds, while formulations \texttt{conic} and \texttt{conic+} require 9 seconds and 2 seconds to solve the same quantity of instances, respectively, thus resulting in 66x and 300x speed-ups in those instances. }
	\label{fig:performance_synt}
\end{figure}

As we observe in our computations with real datasets (see \S\ref{sec:comp_real}), the results here are not representative of the actual performance of the methods in practice. Therefore, we do not provide detailed computational results in this section. We simply comment on the effect of parameters $\tau$ and $\lambda$: instances with small number of outliers ${\lfloor \tau m\rfloor}$ are much easier to solve (most of instances solved to optimality by \texttt{big-M} correspond to small values of $\tau$), and formulation \texttt{conic} benefits from larger values of regularization $\lambda$ as well. Finally, the continuous relaxation of \texttt{conic+} is strong regardless of the combination of parameters, and most of the instances are solved at the root node. 

\subsubsection{Statistical results}\label{sec:synt_res}

We now present the statistical performance for different methods for parameters $(n,m)\in \{(2,100),(20,100),(20,500)\}$. We omit results for methods \texttt{big-M} and \texttt{conic}, since \texttt{conic+} delivers similar solutions much faster. 
In Table~\ref{tab:lambda001}, we compare the performance of \texttt{lad}, and methods \texttt{conic+}, \texttt{alt-opt} and \texttt{ls+l2} with parameter $\lambda=0.01$. The results for $m=100$ are also summarized in Figure~\ref{fig:synt}. Table~\ref{tab:riskLambda} shows the effect of varying the parameter $\lambda$ for the relative risk of estimators \texttt{conic+}, \texttt{alt-opt} and \texttt{ls+l2}.

\begin{figure}[!ht]
	\centering
\subfloat[$n=2$, breakdown]{\includegraphics[width=0.49\textwidth,trim={13cm 5.5cm 13cm 5.5cm},clip]{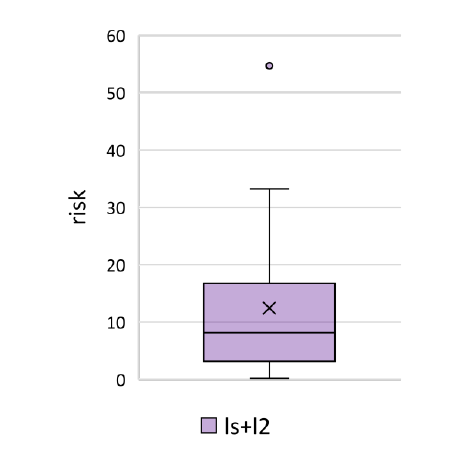}}\hfill\subfloat[$n=2$, robust]{\includegraphics[width=0.49\textwidth,trim={13cm 5.5cm 13cm 5.5cm},clip]{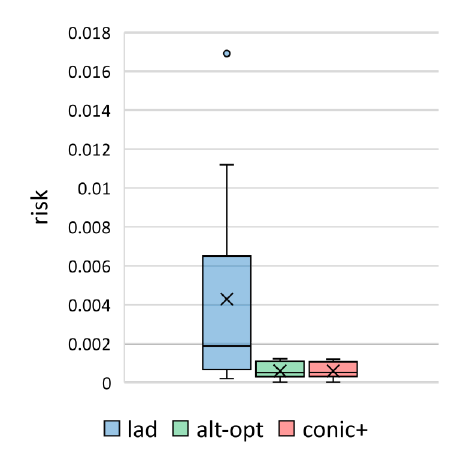}}\hfill\newline %\vskip 5mm 
\subfloat[$n=20$, breakdown]
{\includegraphics[width=0.49\textwidth,trim={13cm 5.5cm 13cm 5.5cm},clip]{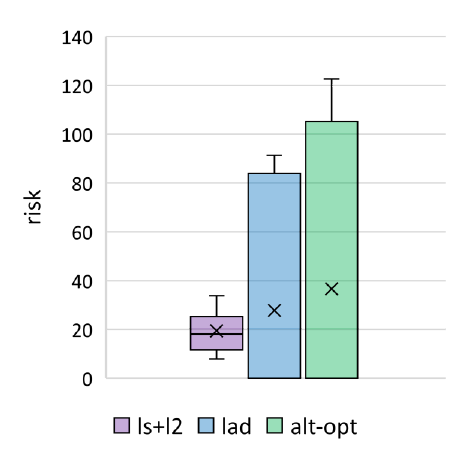}}\hfill\subfloat[$n=20$, robust]{\includegraphics[width=0.49\textwidth,trim={13cm 5.5cm 13cm 5.5cm},clip]{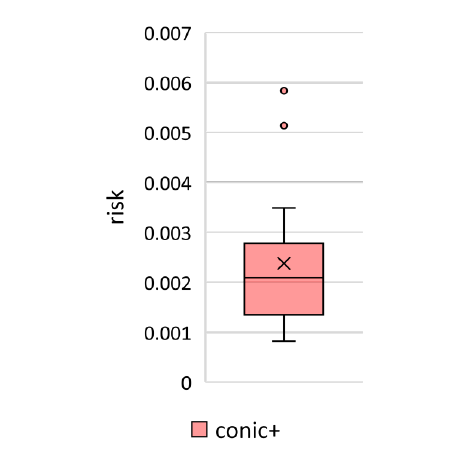}}\hfill\newline %\vskip 5mm 
	\caption{\small Distributions of the relative risk of different methods for instances with $m=100$, $\lambda=0.01$, $\tau\in \{0.1,0.2,0.4\}$. The top row shows instances with $n=2$, and the bottom row shows instances with $n=20$. The left column depicts risk for estimators that broke down, resulting in $\texttt{risk}>1$, and the right column depicts risk for estimators that produced high quality solutions (note the difference in the scale of vertical axis).}
	\label{fig:synt}
\end{figure}

\begin{table}[!ht]
\caption{Statistical performance for different regression methods. We use regularization parameter $\lambda=0.01$ for \texttt{conic+}, \texttt{alt-opt} and \texttt{ls+l2}, while \texttt{lad} have no regularization. Each row represents the average over five instances generated with identical parameters. We show in \textbf{bold} metrics that are the best for that particular class of instances.}
\label{tab:lambda001}
\setlength{\tabcolsep}{2pt}
\begin{center}
\begin{tabular}{ c c  c| c c| c c| c c| c c} 
\hline
\multirow{2}{*}{$\bm{n}$}&\multirow{2}{*}{$\bm{m}$}&\multirow{2}{*}{$\bm{\tau}$}&\multicolumn{2}{c|}{\underline{\texttt{conic+}}}&\multicolumn{2}{c|}{\underline{\texttt{alt-opt}}}&\multicolumn{2}{c|}{\underline{\texttt{lad}}}&\multicolumn{2}{c}{\underline{\texttt{ls+l2}}}\\
&&&\texttt{risk}&\texttt{recall}&\texttt{risk}&\texttt{recall}&\texttt{risk}&\texttt{recall}&\texttt{risk}&\texttt{recall}\\
\hline
2&100&0.1&
\textbf{0.001}&\textbf{ 1.00}&\textbf{0.001}&\textbf{1.00}&\textbf{0.001}&-&6.387&-\\
2&100&0.2&
\textbf{0.001}& \textbf{1.00}&\textbf{0.001}&\textbf{1.00}&0.002&-&7.443&-\\
2&100&0.4&
\textbf{0.001}&\textbf{1.00}&\textbf{0.001}&\textbf{1.00}&0.009&-&23.528&-\\
20&100&0.1&
\textbf{0.001}&\textbf{1.00}&\textbf{0.001}&\textbf{1.00}&0.003&-&12.262&-\\
20&100&0.2&
\textbf{0.002}&\textbf{1.00}&\textbf{0.002}&\textbf{1.00}&0.005&-&19.388&-\\
20&100&0.4&
\textbf{0.004}&\textbf{1.00}&109.979&0.65&83.271 &-&26.580&-\\
20&500&0.1&
\textbf{0.000}&\textbf{1.00}&\textbf{0.000}&\textbf{1.00}&\textbf{0.000} &-&1.631&-\\
20&500&0.2&
\textbf{0.000}&\textbf{1.00}&\textbf{0.000}&\textbf{1.00}&0.001 &-&3.234&-\\
20&500&0.2&
\textbf{0.001}&\textbf{1.00}&\textbf{0.001}&\textbf{1.00}&0.003 &-&4.695&-\\
\end{tabular}
\end{center}
\end{table}

\begin{sidewaystable}
\caption{Relative risk for \texttt{conic+}, \texttt{alt-opt} and \texttt{ls+l2}. Each row represents the average over five instances generated with identical parameters.}
\setlength{\tabcolsep}{2pt}
\label{tab:riskLambda}
\begin{center}
\begin{tabular}{ c c  c| c c c| c c c| c c c| c c c} 
\hline
\multirow{2}{*}{$\bm{n}$}&\multirow{2}{*}{$\bm{m}$}&\multirow{2}{*}{$\bm{\tau}$}&\multicolumn{3}{c|}{\underline{$\lambda=0.01$}}&\multicolumn{3}{c|}{\underline{$\lambda=0.1$}}&\multicolumn{3}{c|}{\underline{$\lambda=0.2$}}&\multicolumn{3}{c}{\underline{$\lambda=0.3$}}\\
&&&\texttt{conic+}&\texttt{alt-opt}&\texttt{ls+l2}&\texttt{conic+}&\texttt{alt-opt}&\texttt{ls+l2}&\texttt{conic+}&\texttt{alt-opt}&\texttt{ls+l2}&\texttt{conic+}&\texttt{alt-opt}&\texttt{ls+l2}\\
\hline
2&100&0.1&
0.001&0.001&6.387&0.009&0.009   &5.316&0.030&0.030&4.420&0.058&0.058&3.739\\
2&100&0.2&
0.001&0.001&7.443&0.011&0.011   &6.332&0.036&0.036&5.382&0.067&0.067&4.645\\
2&100&0.4&
0.001&0.001&23.582&0.015&0.015&19.901&0.051&0.051&16.800&0.096&0.096&14.393\\
20&100&0.1&
0.001&0.001&12.262&0.017&0.017&9.165&0.050&0.050&7.904&0.087&0.096&5.747\\
20&100&0.2&
0.002&0.001&19.388&0.023&0.023&14.848&0.063&0.063&11.609&0.107&0.107&9.410\\
20&100&0.4&
0.004&109.979&26.580&0.070&23.527&20.398&0.115&0.115&15.965&0.178&0.178&12.944\\
20&500&0.1&
0.000&0.000&1.631&0.011&0.011&1.360&0.035&0.035&1.147&0.067&0.067&0.997\\
20&500&0.2&
0.000&0.000&3.234&0.013&0.013&2.676&0.043&0.043&2.228&0.079&0.079&1.900\\
20&500&0.4&
0.001&0.001&4.695&0.022&0.022&3.927&0.067&0.067&3.288&0.117&0.117&2.807\\
\end{tabular}
\end{center}
\end{sidewaystable}

We note that estimator \texttt{conic+} results in the best risk and recall in all instances considered, and is the only estimator which does not break down in instances with $n=20$, $m=100$ and $\tau=0.4$ (i.e., instances with the smallest signal-to-noise ratio and larger number of outliers among those considered). 
We observe that formulation \texttt{ls+l2} in general produces poor solutions, as expected. Indeed, with a breakdown point of $0$, the presence of a single outlier could result in arbitrarily poor solutions, and the instances used contain several outliers. Interestingly, while formulation \texttt{lad} also has a breakdown point of $0$, it produces good solutions in most of the instances considered (although even in those instances, the relative risk can be five times more than the risk of other robust approaches). However, in instances with $n=20$, $m=100$ and $\tau=0.4$, the estimator breaks down and results in extremely poor solutions, worse in fact than those produced by estimator \texttt{ls+l2} which ignores outliers. Heuristic \texttt{alt-opt} matches the performance of \texttt{conic+} in instances where $(n,m,\tau)\neq(20,100,0.4)$ --in fact, the heuristic finds optimal solutions in all these instances--, but fails dramatically in the setting $(n,m,\tau)=(20,100,0.4)$: the solutions produced are poor local minima of \eqref{eq:LTS}, and the statistical properties are in fact worse than \texttt{ls+l2} and \texttt{lad}. 

We see from Table~\ref{tab:riskLambda} that as the regularization parameter $\lambda$ increases, the performance of \texttt{ls+l2} improves (showcasing how the $\ell_2$ regularization induces robustness) but is still substantially worse than \texttt{conic+}. On the other hand, we see that an increase of the regularization parameter results in worse performance for \texttt{conic+}, but the risk remains low for all combinations of regularization tested. Finally we observe that in instances with $(n,m,\tau)=(20,100,0.4)$, an increase of regularization results in much better performance for heuristic \texttt{alt-opt}, and the estimator does not break down if $\lambda\geq 0.2$. However, the resulting risk is larger than the risk of solutions produced by \texttt{conic+} with smaller values of regularization parameter.

\subsection{Experiments with real data}\label{sec:comp_real}

We now discuss experiments with real data. First we present the instances used in \S\ref{sec:real_instances} and the metrics tested in \S\ref{sec:real_metrics}, and then discuss computational times in \S\ref{sec:real_computations} and solution quality in \S\ref{sec:real_quality}.

\subsubsection{Instances}\label{sec:real_instances}

To test the methods we use instances included in the software package ``robust base" \cite{rousseeuw2009robustbase}. Specifically, we select the instances that: \emph{(i)} are regression instances (as opposed to classification), \emph{(ii)} do not have missing data, and \emph{(iii)} are not time series data. The resulting 17 datasets are summarized in Table~\ref{tab:instances}. For each instance, we vary the parameter $\lambda\in \{0.05,0.1,0.2\}$ and the proportion of allowed outliers $m-h\in \{\lfloor 0.1m\rfloor,\lfloor 0.2m\rfloor,\lfloor 0.3m\rfloor,\lfloor 0.4m\rfloor\}$, thus creating 12 different instances for each particular dataset. Note that we separate the instances in nine ``easy" datasets (all satisfying $m\leq 40)$ and eight ``hard" datasets (with $m>40$). This distinction is based on the performance of the \texttt{big-M} formulation: the average time to solve all instances for an easy dataset is less than 15 seconds, whereas for the hard datasets there is at least one instance that could not be solved to optimality within the time limit of 10 minutes.

\begin{table}[!ht]
\caption{Real datasets used. Problems in ``easy" datasets can be solved to optimality by \texttt{big-M} in seconds, whereas time limits are reached for \texttt{big-M} in hard datasets.}
\label{tab:instances}
\setlength{\tabcolsep}{2pt}
\begin{center}
\begin{tabular}{ c | c  c c} 
\hline
& \textbf{name} & $\bm{n}$ & $\bm{m}$\\
\hline
\multirow{9}{*}{Easy}&pension&1&18\\
&phosphor&2&18\\
&salinity&3&18\\
&toxicity&9&18\\
&pilot&1&20\\
&wood&5&20\\
&steamUse&4&38\\
&bushfire&4&38\\
&starsCYG&1&41\\
\hline
\multirow{8}{*}{Hard}&alcohol&6&44\\
&education&4&50\\
&epilepsy& 9 &59\\
&pulpfiber&7&62\\
&wagner&6&63\\
&milk&7&86\\
&foodstamp&3&150\\
&radarimage&4&1,573\\
\hline
\end{tabular}
\end{center}
\end{table}

\subsubsection{Metrics}\label{sec:real_metrics}

On real data, there is no ``ground truth" concerning which points are outliers or the actual values of the regression coefficients. Thus, we limit our comparisons to the performance of MIO algorithms (as measured by time, nodes and optimality gap) and the quality of the solutions obtained in terms of the objective value of \eqref{eq:LTS_MIOCubic} of \texttt{conic+} and \texttt{alt-opt}.

\subsubsection{Computational performance of MIO}\label{sec:real_computations}

Similarly to results with synthetic instances, heuristics such as \texttt{alt-opt} run in a fraction of a second in all cases. In computations with easy datasets, \texttt{big-M} solves the instances in three seconds on average, and under 45 seconds in all cases. Formulation \texttt{conic} also requires three seconds on average as well (and 87 seconds in the worst-case), while formulation \texttt{conic+} requires only one second (and under 10 seconds in the worst case). Note that easy datasets have $m\leq 41$, and full enumeration may be possible in most of the instances. In the interest of shortness, we do not present detailed results on the computations with easy datasets, and focus in this section in the more interesting computations of \texttt{MIO} methods with hard datasets. 

Figure~\ref{fig:performance} presents aggregated results for instances with hard datasets. In particular, it shows the percentages of instances solved by methods \texttt{big-M}, \texttt{conic} and \texttt{conic+} within any given time limit. Observe that the performance of all methods in instances with real datasets is worse than the one reported in synthetic instances, despite real datasets being in some cases smaller by an order-of-magnitude. This discrepancy of performance serves as compelling evidence that synthetic instances should not be used to evaluate the ``scalability" of MIO methods for {\tt{LTS}} or related problems in regression. Indeed, as is well-known in the MIO literature, size of an instance is often a poor proxy of its difficulty. 

We see that the \texttt{big-M} formulation struggles, solving only 22\% of the instances within the time limit of 600 seconds. Formulation \texttt{conic+} is better across the board, requiring only 16 seconds to solve 22\% of the instances, and managing to solve 35\% of the instances overall. Formulation \texttt{conic+} is worse than \texttt{conic} in the simpler instances, but much better in the more difficult ones, managing to solve over 45\% of the instances. Indeed, in instances that can be solved easily by other methods, the additional cost of solving SDPs hurts the performance, but the stronger relaxation pays off in difficult instances.

\begin{figure}[!ht]
	\centering
\includegraphics[width=0.80\textwidth,trim={11cm 6cm 11cm 6cm},clip]{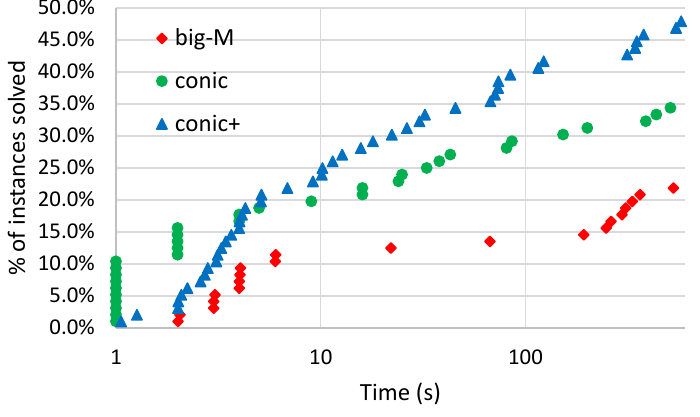}
	\caption{\small Percentage of instances with ``hard" real datasets solved as a function of time (in log scale). Method \texttt{big-M} can solve 22\% of instances in 600 seconds, while formulations \texttt{conic} and \texttt{conic+} require 16 seconds and 7 seconds to solve the same quantity of instances, respectively, thus resulting in 38x and 86x speed-ups in those instances. }
	\label{fig:performance}
\end{figure}

Table~\ref{tab:resultsBudget} presents detailed results for each dataset as a function of the budget parameter $m-h$, and Table~\ref{tab:resultsRegularization} presents detailed results as a function of the regularization parameter $\lambda$. It shows for each dataset and budget/regularization parameter the average time and branch-and-bound nodes used by the solver (time-outs are counted as 600 seconds), as well as the end gaps as reported by the solver (instances solved to optimality count as 0\%). We see that formulation \texttt{big-M} can solve instances if the parameter $m-h$ is small (and the number of feasible solutions $m \choose m-h$ is small as well) but struggles in other instances. Formulations \texttt{conic} and \texttt{conic+} also perform better if the parameter $m-h$ is small (since enumeration is more effective) or if the regularization parameter is large (since the relaxations are stronger). Formulation \texttt{conic} is competitive or better than \texttt{conic+} in the smaller datasets such as alcohol, but \texttt{conic+} is superior overall. 

\begin{table}[!ht]
\caption{Performance of MIO methods as a function of the budget $m-h$. Each row is an average over three instances, with different values of parameter $\lambda$. \texttt{conic+} encountered  numerical difficulties in an instance of radarimage with parameter 
%$m-h=0.4$ 
$m-h=\lfloor 0.4 m\rfloor$ and $\lambda=0.2$, which is  indicated with a $\dagger$ in the table.}
\label{tab:resultsBudget}
\setlength{\tabcolsep}{1pt}
\begin{center}
\begin{tabular}{  c c | c c c| c c c | c c c} 
\hline
\multirow{2}{*}{\textbf{name}} & \multirow{2}{*}{$\bm{m-h}$}& \multicolumn{3}{c|}{\underline{\texttt{big-M}}}& \multicolumn{3}{c|}{\underline{\texttt{conic}}} &\multicolumn{3}{c}{\underline{\texttt{conic+}}}\\
&&\texttt{time}&\texttt{nodes}&\texttt{gap}&\texttt{time}&\texttt{nodes}&\texttt{gap}&\texttt{time}&\texttt{nodes}&\texttt{gap}\\
\hline
\multirow{4}{*}{alcohol}&$\lfloor0.1 m \rfloor$ &2&37,063&0.0\%&1&3,880&0.0\%&3&5,481&0.0\%\\
&$\lfloor 0.2m \rfloor$ &174&2,226,198&0.0\%&5&20,999&0.0\%&11&34,757&0.0\%\\
&$\lfloor 0.3 m \rfloor$ 
&600&9,415,411&7.3\%&12&40,484&0.0\%&10&25,980&0.0\%\\
&$\lfloor 0.4 m \rfloor$ 
&576&8,681,299&7.8\%&7&24,367&0.0\%&14&40,210&0.0\%\\
\hline
\multirow{4}{*}{education}&$\lfloor 0.1 m \rfloor$  
&4&65,793&0.0\%&1&2,841&0.0\%&2&1,288&0.0\%\\
&$\lfloor 0.2 m \rfloor$ 
&600&5,015,602&35.1\%&49&73,475&0.0\%&13&21,019&0.0\%\\
&$\lfloor 0.3 m \rfloor$ 
&600&4,037,385&60.4\%&571&601,629&8.9\%&143&179,170&0.0\%\\
&$\lfloor 0.4 m \rfloor$ 
&600&2,961,573&60.0\%&600&772,419&19.6\%&256&142,113&4.4\%\\
\hline
\multirow{4}{*}{epilepsy}&$\lfloor 0.1 m \rfloor$ 
&10&43,399&0.0\%&1&1,382&0.0\%&3&510&0.0\%\\
&$\lfloor 0.2 m \rfloor$ 
&424&4,425,843&1.1\%&26&70,898&0.0\%&6&9,211&0.0\%\\
&$\lfloor 0.3 m \rfloor$ 
&600&5,396,641&29.1\%&600&1,705,862&11.1\%&249&215,514&1.7\%\\
&$\lfloor 0.4 m \rfloor$ 
&600&6,885,941&29.8\%&600&1,922,058&18.6\%&351&519,278&2.1\%\\
\hline
\multirow{4}{*}{pulpfiber}&$\lfloor 0.1 m \rfloor$ 
&5&49,645&0.0\%&2&4,385&0.0\%&3&4,219&0.0\%\\
&$\lfloor 0.2 m \rfloor$ 
&600&2,907,040&26.5\%&397&565,807&2.3\%&414&591,617&0.0\%\\
&$\lfloor 0.3 m \rfloor$ 
&600&2,881,002&35.7\%&600&839,198&13.0\%&600&562,363&11.8\%\\
&$\lfloor 0.4 m \rfloor$ 
&600&3,161,658&36.0\%&600&348,519&16.7\%&600&553,325&12.5\%\\
\hline
\multirow{4}{*}{wagner}&$\lfloor 0.1 m \rfloor$ 
&292&3,080,940&0.0\%&225&464,437&0.0\%&28&68,909&0.0\%\\
&$\lfloor 0.2 m \rfloor$ 
&600&5,003,290&60.1\%&600&873,538&36.9\%&428&737,124&12.2\%\\
&$\lfloor 0.3 m \rfloor$ 
&600&3,249,964&65.9\%&600&553,961&48.1\%&600&313,817&21.4\%\\
&$\lfloor 0.4 m \rfloor$ 
&600&3,263,027&63.1\%&600&730,974&53.0\%&600&279,949&29.3\%\\
\hline
\multirow{4}{*}{milk}&$\lfloor 0.1 m \rfloor$ 
&600&2,166,874&52.6\%&600&664,593&21.4\%&424&414,886&6.8\%\\
&$\lfloor 0.2 m \rfloor$ 
&600&2,334,757&75.6\%&600&615,562&46.7\%&600&283,008&21.9\%\\
&$\lfloor 0.3 m \rfloor$ 
&600&2,342,248&74.7\%&600&601,391&51.4\%&600&263,689&26.1\%\\
&$\lfloor 0.4 m \rfloor$ 
&600&2,457,988&73.4\%&600&453,377&52.4\%&600&336,605&25.0\%\\
\hline
\multirow{4}{*}{foodstamp}&$\lfloor 0.1 m \rfloor$ 
&600&3,679,125&79.7\%&600&962,111&15.6\%&387&652,085&5.4\%\\
&$\lfloor 0.2 m \rfloor$ 
&600&3,428,673&94.1\%&600&937,462&51.1\%&600&628,270&31.0\%\\
&$\lfloor 0.3 m \rfloor$ 
&600&1,365,841&95.4\%&600&232,394&62.7\%&600&260,683&42.2\%\\
&$\lfloor 0.4 m \rfloor$ 
&600&1,568,131&96.2\%&600&250,902&65.9\%&600&260,335&47.0\%\\
\hline
\multirow{4}{*}{radarimage}&$\lfloor 0.1 m \rfloor$ 
&600&117,206&99.7\%&600&18,268&72.7\%&600&4,604&30.8\%\\
&$\lfloor 0.2 m \rfloor$ 
&600&219,463&99.8\%&600&17,128&80.7\%&600&5,066&44.1\%\\
&$\lfloor 0.3 m \rfloor$ 
&600&258,303&99.8\%&600&21,989&87.9\%&600&7,681&53.4\%\\
&$\lfloor 0.4 m \rfloor$ 
&600&255,181&99.9\%&600&19,899&92.1\%&$\dagger$&$\dagger$&$\dagger$\\
\hline
\end{tabular}
\end{center}
\end{table}

\begin{table}[!ht]
\caption{Performance of MIO methods as a function of the regularization $\lambda$. Each row is an average over four instances, with different values of parameter $m-h$. \texttt{conic+} encountered numerical difficulties in an instance of radarimage with parameter 
%$m-h=0.4$ 
$m-h=\lfloor 0.4 m\rfloor$ and $\lambda=0.2$, which is indicated with a $\dagger$ in the table.}
\label{tab:resultsRegularization}
\setlength{\tabcolsep}{1pt}
\begin{center}
\begin{tabular}{  c c | c c c| c c c | c c c} 
\hline
\multirow{2}{*}{\textbf{name}} & \multirow{2}{*}{$\bm{\lambda}$}& \multicolumn{3}{c|}{\underline{\texttt{big-M}}}& \multicolumn{3}{c|}{\underline{\texttt{conic}}} &\multicolumn{3}{c}{\underline{\texttt{conic+}}}\\
&&\texttt{time}&\texttt{nodes}&\texttt{gap}&\texttt{time}&\texttt{nodes}&\texttt{gap}&\texttt{time}&\texttt{nodes}&\texttt{gap}\\
\hline
\multirow{3}{*}{alcohol}&0.05&331&4,998,999&2.7\%&15&49,984&0.0\%&18&52,293&0.0\%\\
&0.10&317&5,014,222&4.0\%&4&12,738&0.0\%&8&22,117&0.0\%\\
&0.20&366&5,256,758&4.7\%&1&4,576&0.0\%&3&5,422&0.0\%\\
\hline
\multirow{3}{*}{education}&0.05&451&2,672,308&32.0\%&321&527,878&12.4\%&244&208,764&3.3\%\\
&0.10&452&3,013,939&38.5\%&311&279,599&7.1\%&51&37,059&0.0\%\\
&0.20&451&3,374,019&46.0\%&284&280,297&1.8\%&15&11,871&0.0\%\\
\hline
\multirow{3}{*}{epilepsy}&0.05&456&4,936,641&12.9\%&310&733,918&9.9\%&303&296,293&2.8\%\\
&0.10&392&3,381,921&17.1\%&307&998,623&7.9\%&125&206,952&0.0\%\\
&0.20&378&4,245,307&15.0\%&304&1,042,610&4.5\%&28&55,140&0.0\%\\
\hline
\multirow{3}{*}{pulpfiber}&0.05&451&2,151,589&26.2\%&451&513,025&11.5\%&445&605,184&8.2\%\\
&0.10&451&2,277,393&24.7\%&410&432,838&6.9\%&388&327,832&6.1\%\\
&0.20&452&2,320,527&22.8\%&339&372,570&5.6\%&379&346,628&4.0\%\\
\hline
\multirow{3}{*}{wagner}&0.05&533&4,175,966&49.8\%&547&751,548&40.9\%&467&477,785&26.0\%\\
&0.10&524&3,663,863&47.6\%&500&593,842&34.7\%&453&335,998&16.0\%\\
&0.20&512&3,108,087&44.4\%&472&621,794&27.9\%&322&236,066&5.1\%\\
\hline
\multirow{3}{*}{milk}&0.05&600&2,366,186&70.0\%&600&594,747&57.2\%&600&276,703&36.4\%\\
&0.10&600&2,322,270&69.8\%&600&603,373&44.2\%&600&395,896&19.1\%\\
&0.20&600&2,317,945&67.3\%&600&553,073&27.5\%&467&301,042&4.4\%\\
\hline
\multirow{3}{*}{foodstamp}&0.05&600&2,489,055&92.2\%&600&579,845&62.9\%&600&466,488&46.5\%\\
&0.10&600&2,466,540&90.9\%&600&552,682&49.4\%&586&518,835&31.0\%\\
&0.20&600&2,575,733&90.9\%&600&654,626&34.2\%&453&365,707&16.8\%\\
\hline
\multirow{3}{*}{radarimage}&0.05&600&220,631&99.8\%&600&22,315&92.1\%&600&9,388&67.2\%\\
&0.10&600&200,154&99.8\%&600&22,358&85.5\%&600&4,670&46.1\%\\
&0.20&600&216,832&99.8\%&600&13,291&72.4\%&$\dagger$&$\dagger$&$\dagger$\\
\hline
\end{tabular}
\end{center}
\end{table}

\subsubsection{Solution quality}\label{sec:real_quality}

We now compare the best solutions found by formulation \texttt{conic+} and heuristic \texttt{alt-opt} in the real datasets. For each instance, we compute the gap of any method as $$\text{Gap}=\frac{\zeta_{\text{method}}-\zeta^*}{\zeta^*},$$
where $\zeta_{\text{method}}$ is the objective value found by the method and $\zeta^*$ is the objective value of the best solution found for that instance (by any method). The results are presented in Figure~\ref{fig:gaps_real}. We see that \texttt{alt-opt} produced worse solutions than \texttt{conic+} in close to 40\% of the instances, and in those instances the gaps are relatively large (9\% on average, and has high as 50\% in some instances). In contrast, \texttt{conic+} delivers worse solutions in only 5.4\% of the instances, and the gaps are relatively small in those instances (2\% on average). We conclude that while \texttt{alt-opt} finds optimal solutions (or at least as good as \texttt{alt-opt}) in a good portion of the instances, it may deliver poor quality solutions when it fails. In contrast, \texttt{conic+} seems to be reliable in all cases (at the expense of additional computational time). 

\begin{figure}[!ht]
	\centering
\subfloat[Optimality gap of \texttt{conic+} in the 5.4\% of instances where \texttt{alt-opt} produced better solutions]{\includegraphics[width=0.45\textwidth,trim={13cm 5.5cm 13cm 5.5cm},clip]{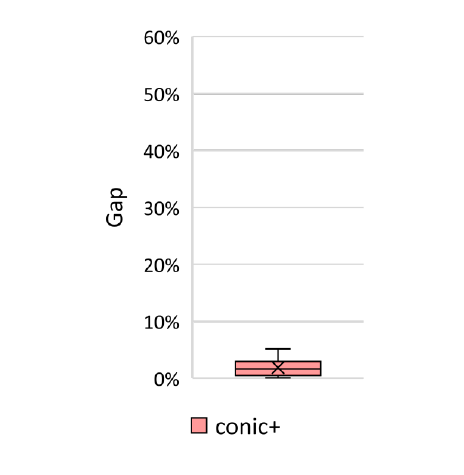}}\hfill\subfloat[Optimality gap of \texttt{alt-opt} in the 38.7\% of instances where \texttt{conic+} produced better solutions]{\includegraphics[width=0.45\textwidth,trim={13cm 5.5cm 13cm 5.5cm},clip]{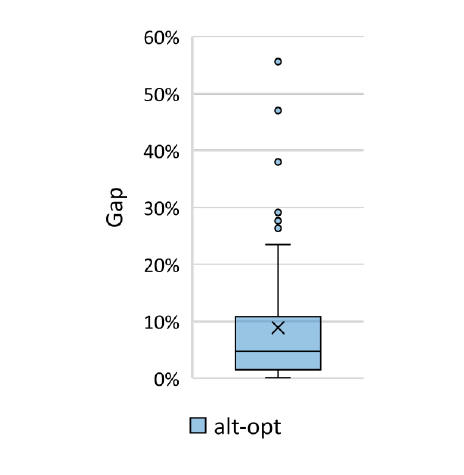}}\hfill
	\caption{\small Optimality gaps of \texttt{conic+} and \texttt{alt-opt} in instances with real datasets where they are outperformed by the other method. Method \texttt{conic+} delivers optimal solutions in most of the instances, and results in small optimality gaps (average 1.9\%) when outperformed. Method \texttt{alt-opt} delivers inferior solutions in more than 1/3 of the instances, with average optimality gaps of (8.9\%), and larger than 25\% in several instances.}
	\label{fig:gaps_real}
\end{figure}

\section{Conclusions}\label{sec:conclusions}

We studied relaxations for a class of mixed-integer optimization problems arising often in statistics. The problems under study are characterized by products of binary variables with nonlinear quadratic terms. Few MIO approaches exist in the literature for the problems considered, and rely on big-M linearizations of the cubic terms, resulting in weak relaxations which provide trivial bounds only. In the paper, we derive the first big-M free relaxations of the problems considered, and our numerical studies with least trimmed squares instances confirm that the suggested relaxations are substantially better than the state-of-the-art. We hope that the study in the paper serves to pave the way for efficient solution of the problems considered via mixed-integer optimization.

\bibliographystyle{abbrv}
\bibliography{references}

\end{document}